\documentclass[final,twoside]{siamltex}

\usepackage{a4wide}
\usepackage{graphicx}
\usepackage{amstext, amsmath, amssymb, amsfonts, amsbsy}

\usepackage{latexsym}
\usepackage{booktabs}
\usepackage{url}
\usepackage{import}
\usepackage{fancyvrb}
\usepackage{stmaryrd}
\usepackage{pdfsync}
\usepackage{algorithm}
\usepackage{tikz}
\usepackage{pgfplots}
\usepackage{pgfplotstable}
\usepackage{subfig}
\usepackage{tabularx}
\usepackage{snapshot}

\usepackage[shadow]{todonotes}
\usepackage[numbers]{natbib}

\newtheorem{remark}[theorem]{Remark}

\numberwithin{equation}{section}
\numberwithin{figure}{section}
\numberwithin{table}{section}
\numberwithin{theorem}{section}

\newcommand{\R}{\mathbb{R}}

\newcommand{\V}{\mathbb{V}}
\newcommand{\foralls}{\forall\,}

\newcommand{\ds}{\,\mathrm{d}s}

\newcommand{\mesh}{\mathcal{T}_h}

\newcommand{\bscp}[1]{\left( #1 \right)}

\DefineVerbatimEnvironment{code}{Verbatim}{frame=single,rulecolor=\color{blue}}

\newcommand{\bfsigma}{\sigma}
\newcommand{\bfu}{{u}}
\newcommand{\bff}{{f}}
\newcommand{\bfg}{{g}}
\newcommand{\bfv}{{v}}

\newcommand{\bfn}{{n}}

\newcommand{\bfeps}{{\epsilon}}

\newcommand{\bftau}{{{ \tau}}} 

\newcommand{\Oast}{\Omega^{\ast}}

\newcommand{\pO}{\Gamma}

\newcommand{\Fast}{\mathcal{F}_{\Gamma}}

\newcommand{\mcF}{\mathcal{F}}

\newcommand{\mcV}{\mathcal{V}}

\newcommand{\mcC}{\mathcal{C}}
\newcommand{\mcO}{\mathcal{O}}

\newcommand{\tn}{|\mspace{-1mu}|\mspace{-1mu}|}

\newcommand{\nablan}{\partial_{\bfn}}

\newcommand{\restr}[2]{ \left. #1 \right|_{#2}}
\newcommand{\tnorm}[1]{\interleave #1\interleave}
\newcommand{\jump}[1]{\left\llbracket #1\right\rrbracket}
\newcommand{\norm}[1]{\left\| #1\right\|}
\newcommand{\snorm}[1]{\left| #1\right|}

\title{\bf A Stabilized Cut Finite Element Method for the Three Field
Stokes Problem}
\author{Erik Burman\footnotemark[2] \and  Susanne Claus\footnotemark[2] \and Andr\'e Massing\footnotemark[3]}

\begin{document}
\maketitle
\renewcommand{\thefootnote}{\fnsymbol{footnote}}

\footnotetext[2]{Department of Mathematics, University College
  London, Gower Street, London WC1E 6BT. email: e.burman@ucl.ac.uk, susanne.claus@ucl.ac.uk}
\footnotetext[3]{Simula Research Laboratory, P.O. Box 134, 1325
Lysaker. email: massing@simula.no}

\renewcommand{\thefootnote}{\arabic{footnote}}

\begin{abstract}
  \noindent 
  We propose a Nitsche-based fictitious domain method for
  the three field Stokes problem in which the boundary of the domain
  is allowed to cross through the elements of a fixed background mesh.
  The dependent variables of velocity, pressure and extra-stress
  tensor are discretised on the background mesh using linear finite
  elements. This equal order approximation is stabilized using a
  continuous interior penalty (CIP) method.  On the unfitted domain
  boundary, Dirichlet boundary conditions are weakly enforced using
  Nitsche's method. We add CIP-like ghost penalties in the boundary
  region and prove that our scheme is inf-sup stable and that it has
  optimal convergence properties independent of how the domain
  boundary intersects the mesh. Additionally, we demonstrate that the
  condition number of the system matrix is bounded independently of
  the boundary location.  We corroborate our theoretical findings numerically.  
\end{abstract}



\section{Introduction}
In this article, we develop a stabilized finite element method for the
so-called three field Stokes system \cite{BonvinPicassoStenberg2001,Bonito2006,Bonito2008}. 
In the three field Stokes equation, the
extra-stress tensor is considered as a separate variable, additionally to velocity and pressure. 
This description of the Stokes system is of particular
interest in viscoelastic fluid mechanics, where the extra-stress
tensor is related to the rate of deformation tensor through a
non-linear constitutive equation and therefore can no longer be easily
substituted into the momentum equation \cite{Owens2002}. \\
In particular, we develop a finite element scheme in which the surface of the fluid can cut elements in the computational mesh in an arbitrary manner.
Such cut finite element methods are especially beneficial for
applications in which the use of interface tracking techniques such as
arbitrary Lagrangian Eulerian methods \cite{DoneaALE}, where the mesh is fitted and moved with the interface, involve frequent re-meshing and
sophisticated mesh moving algorithms that can be prohibitively
expensive. One such application of high interest is the simulation of
viscoelastic free surface flows in which the fluid surface undergoes
large deformations and in which drop detachments may occur. 
This type of free surface flow of viscoelastic liquids plays a key
role in a wide range of industrial applications such as surface
coating for molten plastics, filtration operations of engine oils or
inkjet printing. 
\\
\noindent Our cut finite element method is based on an earlier formulation for fitted meshes presented in \cite{Bonito2008,Bonito2006}, where equal order
approximation spaces for all variables are combined with a
continuous interior penalty method to obtain a stable and optimally
convergent method for the three field Stokes equation.
We employ Nitsche's method \cite{Nitsche1971}
to weakly enforce the boundary conditions on the unfitted boundary
domain.  Nitsche-type methods for unfitted interface problems and fictitious domain
methods have previously been developed in
\cite{HansboHansbo2002,HansboHansboLarson2003,DoHa09,HaDo10} for elliptic problems, in
\cite{HaHa04,BeckerBurmanHansbo2009} for elasticity problems and in
\cite{MassingLarsonLoggEtAl2013,MassingLarsonLoggEtAl2013a,BurmanHansbo2013} for Stokes problems.
A particular complication here is the potential occurrence of
elements which are only partially covered by the physical domain
$\Omega$. In such a case it has been demonstrated in
\cite{BurmanHansbo2012,Burman2010} that the sole application of Nitsche's method results in suboptimal schemes, where the discretization error
and the condition number of the discrete system are highly dependent
on the position of the boundary with respect to the mesh.
In our contribution, we therefore apply so-called ghost penalties
\cite{BurmanHansbo2012,Burman2010} in the vicinity of the boundary
to extend the solution of velocity, pressure and extra-stress from
the cut part of the element partially covered by the physical domain to the
whole element in the interface zone resulting in a fictitious domain
approach.  
The ghost penalties consist of penalties on the gradient jumps of
the velocity, pressure and extra-stress and are applied to all edges in the
interface cell layer.
Due to their similarity to the continuous interior penalty terms, our
scheme allows to cope with both inf-sup and fictitious domain related
instabilities in a unified way.
The resulting formulation is weakly consistent and
we prove that our scheme satisfies a uniform inf-sup condition
and exhibits optimal convergence order independent of how the boundary cuts the mesh. 
Numerical experiments demonstrate that the resulting discrete system
is insensitive to the position of the boundary within the
computational domain.

The remainder of this paper is organised as follows.  Firstly, we
briefly review the strong and weak formulation of the three field
Stokes system in Section~\ref{sec:tfstokes}.  In
Section~\ref{sec:fem-formulation}, the novel cut finite element method
for the three field Stokes problem is introduced.
Section~\ref{sec:approximation-properties} summarizes some useful
inequalities and provides certain interpolation estimates which are
necessary to proceed with the stability analysis in
Section~\ref{sec:stability} and to establish the a priori estimates in
Section~\ref{sec:apriori-analysis}. Finally, we present
numerical results corroborating the theoretical findings along with
numerical investigations of the conditioning of the discrete system.
Moreover, we demonstrate the applicability of our discretization method
to complex 3D geometries.

\section{The three field Stokes problem}
\label{sec:tfstokes}
Let $\Omega$ be a bounded domain in $\mathbb{R}^d$ ($d=2$ or $3$) with
a Lipschitz boundary $\Gamma = \partial \Omega$. The three field
Stokes problem reads: Find the extra-stress tensor $\bfsigma: \Omega
\rightarrow
  \mathbb{R}^{d \times d}$, the velocity $\bfu: \Omega \rightarrow
  \mathbb{R}^{d}$ and the pressure $p: \Omega \rightarrow
  \mathbb{R}$ such that
  \begin{equation}
  \left\{
  \begin{aligned}
      \bfsigma - 2 \eta \bfeps(\bfu) &= 0 \mbox{ in } \Omega,
    \\
      - \nabla \cdot \bfsigma + \nabla p &= \bff \mbox{ in } \Omega,
    \\
      \nabla \cdot \bfu &= 0 \mbox{ in } \Omega, 
    \\
      \bfu &= \bfg \mbox{ on } \partial \Omega.
  \end{aligned}
    \label{eq:tfstokes-strong}
    \right.
  \end{equation}
Here, $\bfeps(\bfu) = \frac{1}{2} \left(\nabla \bfu + \nabla
  \bfu^{\top}\right)$ is the rate of deformation tensor,
  $\bff:\Omega \rightarrow \mathbb{R}^d$ is the body force,
  $\eta$ is the fluid viscosity and $ \bfg$ is the Dirichlet boundary
  value.
To be compatible with the divergence constraint
in~\eqref{eq:tfstokes-strong}, the boundary data is supposed to satisfy
$\int_{\partial \Omega} \bfn \cdot \bfg \ds = 0$ where $\bfn$ denotes
the outward pointing boundary normal.
Note that the only difference between the Stokes equation and the
three field Stokes equation is that the extra-stress tensor is kept as
a separate variable. This type of equation system is of particular
interest in viscoelastic fluid mechanics, where the stress tensor
depends on the rate of deformation tensor through a non-linear
constitutive equation. Hence, the extra-stress tensor can no longer be
substituted into the momentum equation. \\ The weak formulation of the
three field Stokes system is obtained by multiplying
\eqref{eq:tfstokes-strong} with test functions $(\bftau,\bfv,p) \in
\left[L^2(\Omega)\right]^{d \times d} \times \left[H^1_0(\Omega)\right]^{d} \times L^2_0(\Omega)$, by integrating over $\Omega$,
and by integrating by parts. The resulting weak formulation reads:
Find $(\bfsigma, \bfu, p) \in \left[L^2(\Omega)\right]^{d \times d }
\times\left[H^1_0(\Omega)\right]^{d} \times
L^2_0(\Omega) $ such that 
\begin{eqnarray}
\frac{1}{2\eta}
\left(\bfsigma, \, \bftau \right)_{\Omega} - \left( \bfeps(\bfu), \,
\bftau \right)_{\Omega} &=& 0 \nonumber,
  \\
\left(\bfsigma, \, \bfeps(\bfv) \right)_{\Omega} - \left(p, \, \nabla
\cdot \bfv \right)_{\Omega} &=& \left(\bff, \, \bfv \right)_{\Omega}
  \nonumber, \\
\left(\nabla \cdot \bfu, \, q \right)_{\Omega} &=&  0 \nonumber
  \\
\label{eq:tfstokes-weak}
\end{eqnarray}
for all $(\bftau, \bfv, q) \in \left[L^2(\Omega)\right]^{d \times d}
\times \left[H^1_g(\Omega)\right]^{d} \times L^2_0(\Omega)$.
Here and throughout this work, we use the notation
$H^s(U)$ and $[H^s(U)]^d$ for the standard Sobolev space of 
order $s \in \R$ and their $\R^d$-valued equivalents defined on the 
(possibly lower-dimensional) domain $U \subseteq \R^d$.
The associated inner products and norms are written as $(\cdot, \cdot)_{s, U}$
and $\| \cdot \|_{s, U}$. If $s = 0$, we usually drop the index $s$ if no
ambiguities occur.
For $s > 1/2$, we let $[H_g^s(\Omega)]^d$ consist of all functions in
$[H(\Omega)]^d$ whose boundary traces are equal to $g$.
Finally, $L^2_0(\Omega)$ denotes the functions in $L^2(\Omega)$ with zero
average over $\Omega$.

\section{Stabilized Cut Finite Element Formulation}
\label{sec:fem-formulation}
Let $\Omega$ be an open and bounded domain in $\R^d$ ($d = 2, 3$) with
Lipschitz boundary $\Gamma = \partial \Omega$ and let $\mesh$ be a
quasi-uniform tesselation that covers the domain $\Omega$.
We do not assume that the mesh $\mesh$ is fitted to the boundary of
$\Omega$, but we require that $T \cap \Omega \neq \emptyset \,, \, 
\foralls T \in \mesh$.
Typically, the mesh $\mesh$ can be thought of as a suitable sub-mesh
of a larger and easy to generate mesh,
	see Figure~\ref{fig:computational-domain}. 
The domain $\Oast$ consisting of the union of
all elements $T \in \mesh$ is called the \emph{fictitious domain}. 
For mesh faces in $\mesh$, i.e. edges
of elements in two dimensions and faces in three
dimensions, we distinguish between \emph{exterior faces}, $\mathcal{F}_e$, 
which are faces that belong to one element only and are thus part of the
boundary $\partial \Oast$ and \emph{interior faces}, $\mathcal{F}_i$, which
are faces that are shared by two elements in $\mesh$. 
\begin{figure}[t!]
  \begin{center}
\subfloat[]{\label{subfig: fictitious domain}
    \includegraphics[width=0.55\textwidth]{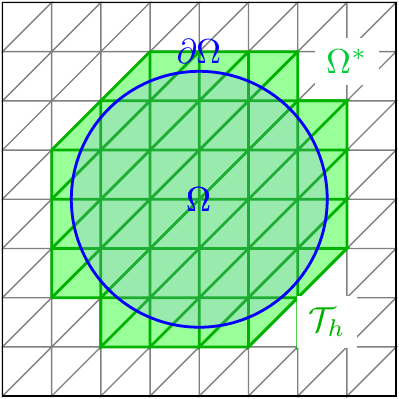}}
\subfloat[]{\label{subfig: face notation}
    \includegraphics[width=0.35\textwidth]{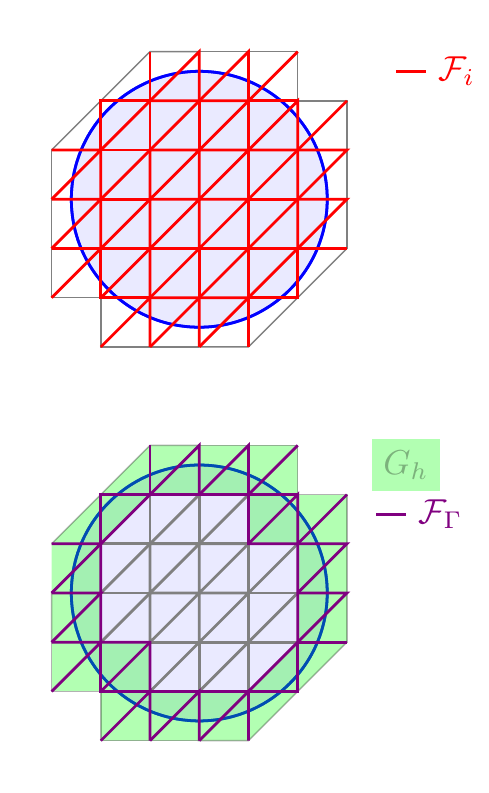}}
  \end{center}
  \caption{Schematics of \protect\subref{subfig: fictitious domain} the computational domain $\Omega$ covered by a fixed and regular background mesh $\mesh$ and the fictitious domain $\Omega^*$ consisting of all elements in $\mesh$ with at least one part in $\Omega$ and  \protect\subref{subfig: face notation} the face notation.}
  \label{fig:computational-domain}
\end{figure}
Next, let $G_h$ be the subset of elements in $\mesh$
that intersect the boundary $\pO$
\begin{equation}
  \label{eq:define-cutting-cell-mesh}
G_h = \{T \in \mesh: T \cap \pO \neq \emptyset \}
\end{equation}
and let us introduce the notation $\mcF_{\Gamma}$ for the set of all interior faces
belonging to elements intersected by the boundary $\pO$
\begin{equation}
  \label{eq:define-cutting-face-mesh}
  \mcF_{\Gamma} = \{ F \in \mcF_i :\;
  T^+_F \cap \pO \neq  \emptyset
  \vee
  T^-_F \cap \pO \neq  \emptyset
  \}.
\end{equation}
Here, $T^+_F$ and $T^-_F$ are the two elements sharing the interior
face $F \in \mathcal{F}_i$. Figure~\ref{subfig: face notation}
illustrates these notations.
To ensure that the boundary $\Gamma$ is reasonably resolved by $\mesh$,
we make the following assumptions:
\begin{itemize}
\item G1: The intersection between $\Gamma$ and a face $F \in
  \mcF_i$ is simply connected; that is, $\Gamma$ does not
  cross an interior face multiple times.
\item G2: For each element $T$ intersected by $\Gamma$, there exists a
  plane $S_T$ and a piecewise smooth parametrization $\Phi: S_T \cap T
  \rightarrow \Gamma \cap T$.
\item G3: We assume that there is an integer $N>0$ such that for each
  element $T \in G_h$ there exists an element $T' \in
  \mesh \setminus G_h$ and at most $N$ elements
  $\{T\}_{i=1}^N$ such that $T_1 = T,\,T_N = T'$ and $T_i \cap T_{i+1}
  \in \mcF_i,\; i = 1,\ldots N-1$.  In other words, the
  number of faces to be crossed in order to ``walk'' from a cut
  element $T$ to a non-cut element $T' \subset \Omega$ is uniformly bounded.
\end{itemize}
Similar assumptions were made in \citep{HansboHansbo2002,BurmanHansbo2012,MassingLarsonLoggEtAl2013a}.
Next, we introduce the continuous linear finite element
spaces
\begin{align}
  \mathcal{V}_h &= \left\{ v_h \in C^0(\Omega): \restr{v_h}{T} \in
    \mathcal{P}_1(T)\, \foralls T \in \mesh \right\}
\end{align}
for the pressure, 
\begin{align}
  \mathcal{V}_h^d &= \left\{\bfv_h \in C^0(\Omega):
    \restr{\bfv_h}{T} \in [\mathcal{P}_1(T)]^d\, \foralls T \in
    \mesh\right\}
\end{align}
for the velocity and 
\begin{align}
  \mathcal{V}_h^{d \times d} &= \left\{ \bfsigma_h \in C^0(\Omega):
    \restr{\mathbf{\bfsigma}_h}{T} \in [\mathcal{P}_1(T)]^{d \times d}\,
    \foralls T \in \mesh \right\}
\end{align}
for the extra-stress tensor. We combine these spaces in the mixed finite element space 
\begin{equation}
  \label{eq:total-approx-space}
  \V_h = \mathcal{V}_h^{d \times d}
  \times \mathcal{V}_h^d \times \mathcal{V}_h.
\end{equation}
When we assume that the computational mesh matches the
domain, the use of equal order approximation spaces for
pressure, velocity and extra-stress tensor does not result in a stable
discretization in the sense of Babu\v{s}ka--Brezzi.
However, a stable discretization can be obtained
by augmenting the variational formulation~\eqref{eq:tfstokes-weak} with suitable stabilization
terms, see for instance \cite{BonvinPicassoStenberg2001},
where Galerkin least-square techniques were employed, 
and \cite{Bonito2008}, where interior penalty operators were
used. 
Remarkably, the same interior penalty techniques have been
proved beneficial in
\cite{BurmanHansbo2010,BurmanHansbo2012,BurmanHansbo2013,MassingLarsonLoggEtAl2013a}
to devise fictitious domain methods which are robust irrespective of
how the boundary cuts the mesh.
In this work, we therefore employ interior penalty operators
to cope with both inf-sup and cut geometry related instabilities
in a unified way.

For a quantity $x$ representing either a scalar, vector or
tensor-valued function on $\Oast$,
the interior penalty operators are defined by
\begin{equation}
  \label{eq:stabilization_operator}
  s_{k}(x,y) = \sum\limits_{F \in \mathcal{F}_i} h^k  
  \left( \jump{\nabla x}_n \jump{\nabla y}_n \right)_F.
\end{equation} 
Here, $\jump{\nabla x}_n$ denotes the normal jump of the quantity $x$
over the face, $F$, defined as $\jump{\nabla x}_n = \left. \nabla x
\right|_{T_F^+} n_F  - \left. \nabla x
\right|_{T_F^-} n_F$, where $n_F$ denotes a unit
normal to the face $F$ with fixed but arbitrary orientation. 
With this notation, the stabilization operators for pressure-velocity
and velocity-extra-stress coupling employed in \cite{Bonito2008} are
then given by
\begin{align}
  \label{eq:stabilization-operator-p}
  s_p(p_h,q_h) &= \frac{\gamma_p}{2 \eta}
  s_{3}(p_h,q_h),
  \\
  \label{eq:stabilization-operator-u}
  s_u(\bfv_h,\bfu_h) &= 2 \eta \gamma_u
  s_{1}(\bfv_h,\bfu_h),
\end{align}
with $\gamma_u$, $\gamma_p$ being positive penalty parameters
to be determined later.
\begin{remark}
  We would like to emphasize that the face-wise contributions in the
  stabilization operator~\eqref{eq:stabilization_operator} are always 
  computed on the \emph{entire} face $F$, even for cut faces where $F
  \cap \overline{\Omega} \not = F$. Following the nomenclature 
  in \cite{Burman2010,BeckerBurmanHansbo2009}, we refer to such penalties evaluated outside the physical
  domain as ghost-penalties. Such ghost-penalties have been proven to
  be crucial in extending classical finite elements with weakly imposed
  Dirichlet boundary conditions to the fictitious domain case.
  In that sense, the stabilization serves two purposes: to stabilize 
  the discretisation scheme when using equal order interpolation
  spaces and to make the scheme insensitive to the boundary position
  with respect to the mesh.
\end{remark}

In addition, we define a stabilization operator for the extra-stress
variable
\begin{equation}
  \label{eq:stabilization-operators}
  s_{\bfsigma}(\bfsigma_h,\bftau_h) = \frac{\gamma_{\sigma}}{2\eta}
  \sum_{F \in \Fast} h^3
  (\jump{\nabla \bfsigma_h}_n, \jump{\nabla \bftau_h}_n)_F,
\end{equation} 
which acts \emph{only} on $\Fast$ as opposed to the stabilization
operators $s_u(\cdot,\cdot)$ and $s_p(\cdot,\cdot)$ which act on
the set $\mcF_i$ of all interior faces.
The sole purpose of the additional stabilization $s_{\sigma}(\cdot,\cdot)$ is to ensure the
stability of the discretization with respect to how the boundary cuts
the mesh. Here, $\gamma_{\sigma}$ denotes a positive penalty
parameter.

We are now in the position to formulate our stabilized cut finite
element method for the three field Stokes problem.
Introducing the bilinear forms
\begin{align}
  \label{eq:a_h-definition}
  a_h\left(\bfsigma_h, \bfv_h \right) &=  \left(\bfsigma_h, \,
\bfeps(\bfv_h) \right)_{\Omega} - 
\bscp{ 
\bfsigma_h \cdot \bfn, \, \bfv_h }_{\partial \Omega}, 
\\
\label{eq:b_h-definition}
b_h\left(p_h, \bfv_h \right) &= - \left(p_h, \, \nabla \cdot
    \bfv_h \right)_{\Omega} + 
    \bscp{p_h \bfn, \, \bfv_h }_{\partial \Omega},
  \end{align}
the proposed discretization scheme reads: Find
$U_h:=\left(\bfsigma_h,\bfu_h,p_h\right) \in \V_h$ such that for all
$V_h:=(\bftau_h,\bfv_h,q_h) \in \V_h$
\begin{align}
    A_h(U_h,V_h) + S_h(U_h,V_h) = L(V_h), 
\label{prob:tfstokes-cutfem}
\end{align}
where 
  \begin{align}
  \label{eq:A_h-definition}
    A_h(U_h,V_h) =& 
    \frac{1}{2\eta}  \left(\bfsigma_h, \, \bftau_h \right)_{\Omega} 
    + \frac{\gamma_b \eta}{h} 
    \bscp{
    u_h,\,v_h}_{\partial \Omega} 
    \nonumber
    + a_h\left(\bfsigma_h,  \bfv_h \right) 
    - a_h\left(\bftau_h, \bfu_h \right) 
    \\
    &\phantom{=}
    + b_h\left(p_h, \bfv_h \right) 
    - b_h\left(q_h, \bfu_h \right), 
    \\
  \label{eq:L_h-definition}
  L_h(V_h) =&  (\bff, \bfv_h)_{\Omega} + 
  \bscp{ \bftau_h \cdot \bfn, \, \bfg
  }_{\partial \Omega} 
  - \bscp{ q_h \bfn, \, \bfg
  }_{\partial \Omega} + \frac{\gamma_b \eta}{h} 
  \bscp{
    \bfg ,\,\bfv_h
  }_{\partial \Omega},
  \end{align}
with the stabilization operator 
  \begin{align}
  \label{eq:S_h-definition}
    S_h(U_h,V_h) =& 
    s_{\bfsigma}(\bfsigma_h, \bftau_h)
    + s_u(\bfu_h,\bfv_h) 
    + s_p(p_h,q_h).
  \end{align}
Here, the positive penalty constant $\gamma_b$ arises from the weak
enforcement of Dirichlet boundary conditions through Nitsche's method
and $h=\max_{T\in \mesh} h_T$ is the mesh size, where $h_T$ denotes
the diameter of $T$.
\begin{remark}
  For the sake of keeping the technical details presented
  in this work at a moderate level, we assume for the following a
  priori error analysis that the contributions from the cut elements
  $\{ T \cap \Omega : T \in G_h\}$ and boundary parts $\{ T \cap
  \Gamma : T \in G_h\}$ can be computed exactly. For a thorough
  treatment of variational crimes arising from the discretization of a
  curved boundary in the context of cut finite element methods,
  we refer to \cite{BurmanHansboLarsonEtAl2014}.
\end{remark}

\section{Approximation properties}
\label{sec:approximation-properties}
In this section, we summarize some useful inequalities and
interpolation estimates which are necessary to proceed with the
stability and a priori analysis in Section~\ref{sec:stability}
and~\ref{sec:apriori-analysis}.

\subsection{Norms}
Recalling the notation from Section~\ref{sec:fem-formulation},
we introduce the triple norm
\begin{align}
  \tn U \tn^2 
  &= \dfrac{1}{2\eta} \| \bfsigma \|^2_{\Omega}
  + 2\eta \| \bfeps(\bfu) \|^2_{\Omega}
  + \dfrac{1}{2\eta} \| p \|^2_{\Omega}
  + 2 \eta \gamma_b \|h^{-1/2} \bfu \|^2_{\Gamma}
  \label{eq:triple-norm}
\end{align}
for the three field Stokes problem~\eqref{prob:tfstokes-cutfem}
 and its discrete counterpart
\begin{align}
  \tn U_h \tn_h^2 
  &= \tn U_h \tn^2  
  + S_h(U_h, U_h)
  \label{eq:triple-norm-disc}
\end{align}
which will be instrumental in 
studying the stability and convergence properties of
problem~\eqref{prob:tfstokes-cutfem} in the following sections.

A key point in the definition of the discrete norm $\tn U_h \tn_h$
is that the inclusion of the stabilization terms
allows to reconstruct natural norms for the discrete function $U_h$
defined on the \emph{entire} fictitious domain $\Oast$.
More specifically, \citet{BurmanHansbo2012,MassingLarsonLoggEtAl2013a} proved the following lemma:
\begin{lemma} 
  \label{lem:control-norm-with-ghost-penalties}
  Let $\Omega$, $\Oast$ and $\Fast$ be defined as in
  Section~\ref{sec:fem-formulation}. Then
  for all $v_h \in \mcV_h$ it holds
  \begin{equation}
  \| v_h \|_{\Oast}^2 
  \lesssim
  \| v_h \|_{\Omega}^2
  + \sum_{F \in  \Fast} h_F^3
  (\jump{\nabla v_h}_n, \jump{\nabla v_h}_n)_F
  \lesssim
  \| v_h \|_{\Oast}^2,
  \label{eq:u-fd-norm-est1}
\end{equation}
\begin{equation}
  \| \nabla v_h \|_{\Oast}^2
  \lesssim
  \| \nabla v_h
  \|_{\Omega}^2
  + \sum_{F \in \Fast} h_F
  (\jump{\nabla v_h}_{n}, \jump{\nabla v_h}_n)_F
  \lesssim
  \| \nabla v_h \|_{\Oast}^2.
  \label{eq:u-fd-norm-est}
\end{equation}
\end{lemma}
\noindent Here and throughout, we use the notation $ a \lesssim b$ for $ a
\leqslant C b$ for some generic constant~$C$ which varies with the
context but is always independent of the mesh size $h$.

\subsection{Useful inequalities}
We recall the following trace inequality (for a proof, see e.g. \cite{DipietroErn2011}) for $v \in H^1(\Oast)$
  \begin{equation}
    \norm{v}_{\partial T} 
    \lesssim
    h_T^{-1/2} \norm{v}_{T} +
    h_T^{1/2} \norm{\nabla v}_{T}
    \quad \foralls T \in \mesh.
    \label{eq:trace-inequality}
  \end{equation}
  If the intersection $\Gamma \cap T$ does not coincide with a boundary edge of the mesh and if the intersection is subject to conditions G1)--G3), then
  the corresponding inequality
  \begin{align}
    \norm{v}_{T \cap \partial \Omega} 
    \lesssim
      h_T^{-1/2} \norm{v}_{T} 
      + h_T^{1/2} \norm{\nabla v}_{T} 
    \label{eq:trace-inequality-unfitted}
  \end{align}
 holds (see \citet{HansboHansbo2002}).
We will also use the following
  well-known inverse estimates for $v_h \in \mcV_h$
  \begin{alignat}{3}
    \label{eq:inverse-estimate-I}
    \norm{\nabla v_h}_{T} 
    & \lesssim
    h_T^{-1} 
    \norm{v_h}_{T} & & \quad \foralls T \in \mesh,
    \\ 
    \label{eq:inverse-estimate-II}
    \norm{\nabla v_h \cdot n}_{\partial T} 
    & \lesssim
    h_T^{-1/2} 
    \norm{\nabla v_h}_{T} & & \quad \foralls T \in \mesh.
  \end{alignat}
 For elements $T$ intersected by the
    boundary $\Gamma$, we have
  \begin{alignat}{3}
    \norm{\nabla v_h \cdot n}_{T \cap \Gamma} 
    & \lesssim
    h_T^{-1/2} \norm{\nabla v_h}_{T} & & \quad \foralls T \in
    \mesh,
    \label{eq:inverse-estimate-III}
  \end{alignat}
  which is proven in \cite{HansboHansbo2002}.
  Finally, we recall the well-known Korn inequalities, stating that
  \begin{gather}
    \| \nabla \bfv \|_{\Omega} 
    \lesssim 
    \| \bfeps(\bfv) \|_{\Omega} \quad \foralls \bfv \in
    [H_0^1(\Omega)]^d,
    \label{eq:korn-I}
    \\
    \| \bfv \|_{1,\Omega} 
    \lesssim 
    \| \bfeps(\bfv) \|_{\Omega} 
    + \| \bfv \|_{\Omega}
    \quad \foralls \bfv \in
    [H^1(\Omega)]^d.
    \label{eq:korn-II}
  \end{gather}
  Since the norm~\eqref{eq:triple-norm} incorporates the boundary data of $u$, it will be
  most convenient to work with the following variant
  of Korn's inequality~\eqref{eq:korn-II}:
  \begin{lemma}
    For $\bfv \in [H^1(\Omega)]^d$, it holds that
    \begin{align}
      \| \bfv \|_{1,\Omega} 
      \lesssim 
      \| \bfeps(\bfv) \|_{\Omega} 
      + \| \bfv \|_{\partial \Omega}.
      \label{eq:korn-III}
    \end{align}
    \label{lem:korn-III}
  \end{lemma}
  \begin{proof}
    For the sake of completeness, we provide a short proof here,
    which is established by contradiction. Assuming that~\eqref{eq:korn-III}
    does not hold, we can construct a sequence $\{\bfv_n\}_n$ such that
    $\| \bfv_n \|_{1,\Omega} = 1$ and
    \begin{align}
    \| \bfeps(\bfv_n) \|_{\Omega} 
    + \| \bfv_n \|_{\partial \Omega}
    \leqslant \dfrac{1}{n}.
    \label{eq:korn-contradiction}
    \end{align}
    The compact embedding $H^1(\Omega) \subset\subset L^2(\Omega)$ 
    implies that there is a subsequence $\{\bfv_{n'}\}_{n'}$ which converges in
    the $\| \cdot \|_{0, \Omega}$-norm. Since by construction,
    $\| \bfeps(\bfv_n - \bfv_m) \|_{\Omega} \leqslant \tfrac{1}{n} +
    \tfrac{1}{m}$, we conclude using Korn's
    inequality~\eqref{eq:korn-II} that  $\{\bfv_{n'}\}_{n'}$  is also a Cauchy
    sequence in $[H^1(\Omega)]^d$ with $\bfv_{n'}
    \overset{\|\cdot\|_{1,\Omega}}{\to} \bfv'$
    for some $\bfv' \in [H^1(\Omega)]^d$.
    Due to the boundedness of the trace
    operator $T:[H^1(\Omega)]^d \to [L^2(\Omega)]^d$ and ~\eqref{eq:korn-contradiction}, 
    we observe that $\bfv' \in [H^1_0(\Omega)]^d$. 
    Now Poincar\'e's inequality together with Korn's inequality~\eqref{eq:korn-I}--\eqref{eq:korn-II}
    gives the contradiction
    \begin{align*}
      1 = \| \bfv' \|_{1,\Omega} \lesssim \|\bfeps(\bfv')\|_{\Omega}
      + \| \bfv'
      \|_{\Omega} \lesssim \| \nabla \bfv' \|_{\Omega} 
      \lesssim  \|\bfeps(\bfv')\|_{\Omega} = 0.
    \end{align*}
  \end{proof}

\subsection{Interpolation and projection operators}
Before we construct various interpolation operators $L^2(\Omega) \to
\mcV_h$, we recall that for a Lipschitz-domain $\Omega$, an extension
operator
\begin{equation}
  E: H^s({\Omega}) \rightarrow H^s(\Oast)
\end{equation}
can be defined which is bounded
\begin{equation}
  \label{eq:extension-operator-boundedness}
  \norm{E v}_{s,\Oast} \lesssim \norm{v}_{s,\Omega}, \quad s=0,1,2,
\end{equation}
see \cite{Stein1970}
for a proof. Occasionally, we write $v^* = Ev$. 
Then, for any interpolation operator $I_h: H^s(\Oast) \to \mcV_h$,
we can define its ``fictitious domain'' variant 
$I_h^*: H^s(\Omega) \to \mcV_h$ by simply  requiring that
  \begin{equation}
    I_h^*u = I_h(u^*)
  \end{equation}
for $u \in H^s(\Omega)$.
In particular, we will choose $I_h$ to be the Cl\'ement
and Oswald interpolation operators, 
which we denote by $\mcC_h$ and $\mcO_h$,
respectively (see for instance \cite{ErnGuermond2004}).
Recalling that the standard interpolation estimates 
for the Cl\'ement interpolant
\begin{alignat}{3}
\| v - \mcC_h v \|_{r,T} 
& \lesssim
  h^{s-r}| v |_{s,\omega(T)},
&\quad 0\leqslant r \leqslant s \leqslant 2 \quad &\foralls T\in \mesh,
  \label{eq:interpest0}
  \\
\| v - \mcC_h v \|_{r,F} &\lesssim h^{s-r-1/2}| v |_{s,\omega(T)},
&\quad 0\leqslant r \leqslant s \leqslant 2 \quad &\foralls F\in
  \mcF_i 
  \label{eq:interpest1}
\end{alignat}
hold if $v \in H^s(\Oast)$,
we observe that the extended Cl\'ement interpolant $\mcC_h^*$ satisfies
\begin{alignat}{3}
  \| v^* - \mcC_h^* v \|_{r,\Oast} 
  & \lesssim
  h^{s-r}\| v \|_{s,\Omega},
  & &\quad 0\leqslant r \leqslant s \leqslant 2,
  \label{eq:interpest0-ast}
  \\
  \sum_{F\in \mcF} 
  \| v^* - \mcC_h^* v \|_{r,F} 
  &\lesssim h^{s-r-1/2} \| v \|_{s,\Omega},
  & &\quad 0\leqslant r \leqslant s \leqslant 2
  \label{eq:interpest1-ast}
\end{alignat}
due to the boundedness of the extension
operator~\eqref{eq:extension-operator-boundedness}.
Here, $\omega(T)$ is the set of elements in $\mesh$ sharing at least one vertex with $T$ (for \eqref{eq:interpest0}) and sharing at least one vertex with $F \in \mathcal{F}_i$ (for \eqref{eq:interpest1}), respectively. 
The Oswald interpolation operator $\mcO_h: H^2(\mesh) \to \mcV_h$ 
is of particular use in the context of continuous interior penalty
methods, as  it allows to  control the fluctuation $\nabla v_h -
\mcO_h(\nabla v_h)$ in terms of the stabilization
operator~\eqref{eq:stabilization_operator}. 
More precisely, \citet{BurmanFernandezHansbo2006} proved the following lemma: 
\begin{lemma}
  \label{lem:cip-fluctuation-control}
  For all $v_h \in \mcV_h$
  \begin{equation}
  \label{eq:cip-fluctuation-control}
    \norm{h \left( \nabla v_h - \mcO_h\left(\nabla v_h
    \right)\right)}_{\Oast}^2 
    \lesssim
    s_{3}(v_h,v_h).
  \end{equation}
\end{lemma}
To exploit this control given by the stabilization operators
$s_{3}(\cdot, \cdot)$ in the stability analysis of our fictitious
domain method, we define the \emph{stabilized} approximate
$L^2$-projection
\begin{equation}
  \Pi^*_h : L^2(\Omega) \rightarrow \mcV_h
\end{equation}
by
\begin{equation}
  \left(\Pi_h^* u, v_h \right)_{\Omega}
  + s_{3}\left(\Pi_h^* u, v_h \right)
  =
  \left( u, v_h \right)_{\Omega}.
  \label{eq:stab-projection-def}
\end{equation}
We conclude this section by proving certain approximation properties
of the stabilized $L^2$-projection. We start with the following.
\begin{lemma}[$L^2$ stability of $\Pi_h^*$]
  \label{lem:l2-stability}
  For $u\in L^2(\Omega)$ it holds
  \begin{align}
    \norm{\Pi_h^* u}_{\Omega} \leq \norm{
    u}_{\Omega}, \\
    \norm{\Pi_h^* u}_{\Omega^*} \lesssim \norm{
    u}_{\Omega}.
    \label{eq:l2-stability}
  \end{align}
\end{lemma}
\begin{proof}
Using the property 
 \eqref{eq:stab-projection-def}  of the stabilized $L^2$-projection, a Cauchy-Schwarz inequality  and Lemma~\ref{lem:control-norm-with-ghost-penalties}, we obtain
  \begin{align}
    \norm{\Pi_h^* u}_{\Omega}^2 
    &\leq  \norm{ \Pi_h^* u }_{\Omega}^2  + s_3(\Pi_h^* u, \Pi_h^* u) \nonumber \\
    &= \left(u, \Pi_h^* u \right)_{\Omega}  \nonumber \\
    &\leq  \norm{u}_{\Omega}   \norm{\Pi_h^* u}_{\Omega}. 
  \end{align}
  \begin{align}
\norm{\Pi_h^*
  u }_{\Oast}^2 
    &\lesssim  \norm{ \Pi_h^* u }_{\Omega}^2  + s_3(\Pi_h^* u, \Pi_h^* u).
  \end{align}
\end{proof}

\begin{proposition}
  \label{prop:interpolation-estimate-stabilized-projection}
  Assuming a quasi-uniform triangulation, the stabilized projection
  operator $\Pi_h^*$ satisfies the following approximation property for
  $u \in H^s(\Omega)$
  \begin{equation}
    h^{1/2} \norm{u - \Pi_h^* u}_{\partial \Omega} + \norm{u - \Pi_h^*
      u}_{\Omega} + h \norm{\nabla \left(u - \Pi_h^* u \right)}_{\Omega} \leq C
    h^s \snorm{u}_{s}.
    \label{eq:interpolation-estimate-stabilized-projection}
  \end{equation}
\end{proposition}
\begin{proof}
  We begin by writing  equation~\eqref{eq:interpolation-estimate-stabilized-projection} as $I + II +
  III$.  We first proof the $L^2$-error estimate 
  \begin{equation}
   II = \norm{\Pi_h^* u - u}_{\Omega}  \leq C h^s
    \snorm{u}_{H^s(\Omega)}.
  \end{equation}
  Using Lemma~\ref{lem:control-norm-with-ghost-penalties} and the triangle inequality, we obtain
  \begin{align}
     \norm{\Pi_h^*u - u}_{\Omega}  \lesssim  \norm{\Pi_h^*u - u^*}_{\Oast} &\leq
    \| \underbrace{\Pi_h^*u - \mathcal{C}_h^*u}_{\xi_h} \|_{\Oast}
    + \norm{\mathcal{C}_h^*u - u^*}_{\Oast}.
  \end{align}
 Next,  using the approximation properties of the Cl\'ement operator
  $\mcC_h^{\ast}$~\eqref{eq:interpest0-ast}, it is enough to estimate $\xi_h$ as follows
  \begin{align*}
 \norm{\xi_h}_{\Oast}^2 &\lesssim \left( \xi_h,\, \xi_h\right)_{\Omega} + s_3 \left( \xi_h ,\, \xi_h\right) \nonumber \\
    &= \left( u - \mathcal{C}_h^* u ,\, \xi_h\right)_{\Omega} +
    s_3 \left( \mathcal{C}_h^* u, \, \xi_h\right) \nonumber \\
    &\lesssim \norm{u - \mathcal{C}_h^* u}_{\Oast}
    \norm{\xi_h}_{\Omega}
    + s_3 \left( \mathcal{C}_h^* u, \,  \mathcal{C}_h^* u \right)^{1/2}  s_3 \left( \xi_h, \,  \xi_h \right)^{1/2}  \nonumber \\
    &\lesssim \left( \norm{u - \mathcal{C}_h^* u}_{\Oast}
      +  s_3 \left( \mathcal{C}_h^* u, \,  \mathcal{C}_h^* u
    \right)^{1/2} \right) \norm{\xi_h}_{\Oast}, \nonumber 
  \end{align*}
  where in the first and fourth line, we used  Lemma~\ref{lem:control-norm-with-ghost-penalties} and in the second line, we used the property
  ~\eqref{eq:stab-projection-def} to pass from $\Pi_h^* u$ to $u$.
  Consequently, 
\begin{align*}
   \norm{\xi_h}_{\Oast}  &\lesssim h^s \snorm{u^*}_{H^s(\Oast)} 
      \nonumber 
      \\
    &\lesssim h^s \snorm{u}_{H^s(\Omega)}.
\end{align*}
Next, we will prove
  \begin{equation}
   III =     h \norm{\nabla(u - \Pi_h^* u)}_{\Omega} \lesssim   h \norm{\nabla(u^* - \Pi_h^* u)}_{\Omega^*}  \leq C h^s
    \snorm{u}_{H^s(\Omega)}.
  \end{equation}
Using Lemma~\ref{lem:control-norm-with-ghost-penalties}, the
  approximation properties of the Cl\'ement interpolant, the triangle
  inequality, an inverse inequality in combination with the $L^2$-approximation property 
  of the stabilized projection, the third term can be estimated as follows:
  \begin{align*}
  h \norm{\nabla(u^* - \Pi_h^* u)}_{\Omega^*} &\lesssim  h( \norm{\nabla(u^* -
  \mcC_h^{\ast} u)}_{\Omega^*} + \norm{\nabla(\mcC_h^{\ast} u - \Pi_h^*
      u)}_{\Omega^*})
    \\
    & \lesssim  h (\norm{\nabla(u^* - \mcC_h^{\ast} u)}_{\Oast} +
    h^{-1}\norm{\mcC_h^{\ast} u^* - \Pi_h^* u}_{\Oast})
    \\
    & \lesssim  h (\norm{\nabla(u^* - \mcC_h^{\ast} u)}_{\Oast} +
    h^{-1}(\norm{\mcC_h^{\ast} u - u}_{\Oast} + \norm{u^* - \Pi_h^*
      u}_{\Oast} ))
    \\
    & \lesssim  h^s \snorm{u}_{s,\Omega}.
  \end{align*}
  We conclude the proof by bounding the first term via the trace
  inequality~\eqref{eq:trace-inequality-unfitted}
  \begin{align*}
    h^{1/2} \norm{u - \Pi_h^* u}_{\Gamma} \lesssim \norm{ u^* - \Pi_h^*
      u}_{\Oast} + h \norm{\nabla(u^* - \Pi_h^* u)}_{\Oast}
  \end{align*}
  and using the estimates for $II$ and $III$.
\end{proof}
\begin{corollary}[$H^1$ stability of $\Pi_h^*$]
  \label{lem:h1-stability}
  For $u\in H^1(\Omega)$ it holds
  \begin{equation}
    \norm{\nabla \Pi_h^* u}_{\Omega} \lesssim \norm{\nabla
    u}_{\Omega}.
    \label{eq:h1-stability}
  \end{equation}
\end{corollary}
\begin{proof}
  The desired estimate follows from Proposition~\ref{prop:interpolation-estimate-stabilized-projection}: 
\begin{align}
    \norm{\nabla \Pi_h^* u}_{\Omega} &\leq  \norm{\nabla (\Pi_h^* u - u)}_{\Omega} + \norm{\nabla
    u}_{\Omega} \nonumber \\
&\leq C \norm{\nabla
    u}_{\Omega}  + \norm{\nabla
    u}_{\Omega}. 
\end{align}
\end{proof}

\section{Stability estimates}
\label{sec:stability}
In this section, we prove that the stabilized cut finite element
formulation~\eqref{prob:tfstokes-cutfem} for the three field Stokes
problem~\eqref{eq:tfstokes-strong} fulfills an inf-sup condition in the
Babu\v{s}ka--Brezzi sense.
As a first step, we demonstrate that the pressure
stabilization given in~\eqref{eq:stabilization-operators}
allows to formulate a weakened inf-sup condition
for the pressure-velocity coupling when equal-order interpolation
spaces are employed. Similar estimates have previously been 
stated in \cite{BurmanHansbo2006a} for the classical matching mesh case
and in \cite{BurmanHansbo2013,MassingLarsonLoggEtAl2013a}
for a Nitsche-based fictitious domain formulation for the Stokes
problem. Introducing the discrete velocity norm \begin{align}
  \| v_h \|_{1, h}^2 = \| v_h \|_{1, \Oast}^2 + \gamma_b \|h^{-1/2}
  v_h \|_{\Gamma}^2
  \label{eq:discrete-vel-norm}
\end{align}
for $v_h \in [\mcV_h]^d$, we can state the following
\begin{proposition}
  \label{prop:mod-inf-sup-bh}
  Let $p_h \in \mcV_h$, then there is a constant $c > 0$ such that
  \begin{align}
    \sup_{\bfv_h \in \mcV_h^d\setminus \{0\}} \dfrac{b_h(p_h,\bfv_h)}{\| \bfv_h
    \|_{1,h}}
    \gtrsim
    \| p_h \|_{\Omega} -  c s_{3}(p_h,p_h)^{1/2}. 
    \label{eq:mod-inf-sup-bh}
  \end{align}
\end{proposition}
\begin{proof}
Due to the surjectivity of the divergence operator $\nabla \cdot
  : [H_0^1(\Omega)]^d \to L^2(\Omega)$, there exists a $v^p \in
  [H_0^1(\Omega)]^d$ such that $ \nabla \cdot v^p = p_h$ and $ \| v^p
  \|_{1,\Omega} \lesssim \| p_h \|_{\Omega}$. Setting $v_h =
  \Pi^{\ast}_h v^p$ and using the $H^1$-stability of the
  stabilized $L^2$-projection stated in Lemma \ref{lem:h1-stability},
  we thus obtain
  \begin{align}
    \nonumber
    b_h(p_h, v_h) 
    &= b_h(p_h, v^p) + b_h(p_h, \Pi_h^{\ast} v^p - v^p) 
    \\
    &\gtrsim \| p_h \|_{\Omega} \| v^p \|_{1,\Omega} + 
    b_h(p_h, \Pi_h^{\ast} v^p - v^p).
    \label{eq:b_h-split}
  \end{align}
  Next, we estimate the remaining term in~\eqref{eq:b_h-split}.
  Recalling definition~\eqref{eq:b_h-definition} of $b_h(\cdot, \cdot)$ and integrating
  by parts gives
  \begin{align}
    b_h(p_h, \Pi_h^{\ast} v^p - v^p)
    &= \left( \nabla p_h, \Pi^{\ast}_h v^p - v^p \right)_{\Omega}.
    \label{eq:b_h-ibp}
  \end{align}
  We now exploit the (almost) orthogonality of the stabilized
  $L^2$-projection $\Pi_h^{\ast} : [H^1_0(\Omega)]^d \to \mcV^{d}$  by
  inserting $\mcO_h(\nabla p_h) \in \mcV^{d}$ into~\eqref{eq:b_h-ibp},
  yielding
  \begin{align}
    \nonumber
    b_h(p_h, \Pi_h^{\ast} v^p - v^p)
    &= \left( \nabla p_h - \mcO_h(\nabla p_h), \Pi^{\ast}_h v^p - v^p \right)_{\Omega}
    - s_{3}(\mcO_h(\nabla p_h),  \Pi^{\ast}_h v^p) = I + II.
  \end{align}
  Combining Lemma~\ref{lem:cip-fluctuation-control} with the stability and approximation
  properties of $\Pi_h^{\ast}$, cf.~\eqref{eq:h1-stability} and
  \eqref{eq:interpolation-estimate-stabilized-projection}, 
  the first term can be bounded as
  follows:
  \begin{align}
    I
    &\gtrsim 
    -s_{3}(p_h, p_h)^{1/2} \|h^{-1} (\Pi^{\ast}_h v^p - v^p)
    \|_{\Omega}
    \gtrsim 
    - s_{3}(p_h, p_h)^{1/2} \| v^p \|_{1, \Omega}
    \label{eq:b_h-I-estimate}
  \end{align}
  To estimate $II$, recall the definition of $\bfv_h$ and apply
  successively a Cauchy-Schwarz inequality and
  Lemma~\ref{lem:control-norm-with-ghost-penalties}
  to obtain
  \begin{align}
    \nonumber
   II
   \gtrsim - s_{1}(v_h,v_h)^{1/2} s_{5}\left(\mcO_h(\nabla
   p_h),\mcO_h(\nabla p_h)\right)^{1/2}
   &\gtrsim 
   - \| v_h \|_{1,\Oast} 
   s_{5}\left(\mcO_h(\nabla p_h),\mcO_h(\nabla p_h)\right)^{1/2}.
  \end{align}
  Using successively the discrete trace
  inequality~\eqref{eq:inverse-estimate-II}, 
  the inverse inequality~\eqref{eq:inverse-estimate-I} and
  Lemma~\ref{lem:cip-fluctuation-control}, 
  the last term can be bounded in the following way:
  \begin{align}
    \nonumber
   s_{5}\left(\mcO_h(\nabla p_h),\mcO_h(\nabla p_h)\right)
   &= 
   s_{5}\left(\mcO_h(\nabla p_h) - \nabla p_h,
   \mcO_h(\nabla p_h) - \nabla p_h \right)
   \\
    \nonumber
   & \lesssim
   \sum_{T \in \mesh} h^4 \| \nabla (\mcO_h(\nabla p_h) - \nabla p_h )
   \|_T
   \\
   \label{eq:b_h-inconsistency-estimate}
   & \lesssim
   \sum_{T \in \mesh} h^2 \| \mcO_h(\nabla p_h) - \nabla p_h \|_T
   \lesssim
   s_{3}\left(p_h, p_h \right).
  \end{align}
  Consequently, 
  \begin{align}
  II \gtrsim - \| v_h \|_{1,\Oast} s_{3}(p_h,p_h)^{1/2}.
   \label{eq:b_h-II-estimate}
  \end{align}
  Combining ~\eqref{eq:b_h-ibp}, \eqref{eq:b_h-I-estimate}
  and~\eqref{eq:b_h-II-estimate}, we find that for some constants
  $c_1$ and $c_2$
  \begin{align}
    b_h(p_h, v_h) 
    \gtrsim \left(
      \|p_h\|_{\Omega} 
      - c_1 s_{3}(p_h,p_h)^{1/2}
    \right) \| v^p \|_{1,\Omega} 
    - c_2 s_{3}(p_h,p_h)^{1/2} \| v_h \|_{1,\Oast}. 
    \label{eq:b_h-estimate-final}
  \end{align}
  To conclude the proof, we note that since $v_p \in [H^1_0(\Omega)]^d$, we have
  \begin{align*}
    \|v_h \|_{1, h}^2 
    &= \| v_h \|_{1, \Oast}^2 + h^{-1}\| v_h - v_p \|_{\Gamma}^2
    \lesssim
    \| v_p \|_{1, \Omega}^2,
  \end{align*}
  thanks to the stability~\eqref{eq:l2-stability} of the operator
  $\Pi^{\ast}_h$,  the interpolation
  estimate~\eqref{eq:interpolation-estimate-stabilized-projection} 
  applied for $s = 1$ and our choice of $v_h$. As a result,
  \begin{align}
    \dfrac{b_h(p_h, v_h)}
    {\| v_h \|_{1,h}}
    \gtrsim 
    \dfrac{b_h(p_h, v_h)}
    {\| v^p \|_{1,\Omega}}
    \gtrsim 
      \|p_h\|_{\Omega} 
      - c s_{3}(p_h,p_h)^{1/2},
      \label{eq:b_h-estimate-final-II}
  \end{align}
  if $b_h(p_h, v_h) \geqslant 0$, 
  otherwise we can simply use $\widetilde{v}_h = - v_h$ 
  to arrive at~\eqref{eq:b_h-estimate-final-II} with $v_h$
  replaced by $\widetilde{v}_h$.
\end{proof}

As a second step, we state and prove a  
weakened inf-sup condition
for the coupling between the velocity and the extra-stress.
Here, the ``defect'' of the inf-sup condition is quantified
in terms of the velocity stabilization form
in~\eqref{eq:stabilization_operator} and 
the boundary penalization in~\eqref{eq:A_h-definition}.
\begin{proposition}
    \label{prop:mod-inf-sup-ah}
  Let $\bfu_h \in \mcV_h^d$, then there is a constant $c > 0$ such that
  \begin{align}
    \sup_{\bftau_h \in \mcV_h^{d\times d}\setminus \{0\}}
    \dfrac{a_h(\bftau_h,\,\bfu_h)}{\| \bftau_h \|_{\Oast}}
     \gtrsim 
    \| \bfu_h \|_{1,\Omega} -  c 
    \left( s_{1}(\bfu_h,\bfu_h)^{1/2} +
    \bscp{h^{-1} \bfu_h, \bfu_h}^{1/2}_{\Gamma}
    \right).
    \label{eq:mod-inf-sup-ah}
  \end{align}
\end{proposition}
\begin{proof}
  Choose $\bftau_h = \Pi_h^{\ast} \bfeps(\bfu_h)$. Then, by
  adding and subtracting $\bfeps(\bfu_h)$
  and then $\bscp{h^{-1} \bfu_h, \bfu_h}_{\Gamma}^{1/2}$, we obtain
    \begin{align}
      \nonumber
        a_h( \bftau_h ,\, \bfu_h) 
        &= (\Pi^{\ast}_h \bfeps(\bfu_h),\, \bfeps(\bfu_h))_{\Omega}
        - \bscp{\Pi^{\ast}_h \bfeps(\bfu_h) \cdot \bfn,
        \bfu_h}_{\Gamma}
        \\
      \nonumber
        &= \| \bfeps(\bfu_h) \|_{\Omega}^2
            +\left(\Pi^{\ast}_h \bfeps(\bfu_h) -
            \bfeps(\bfu_h),\,\bfeps(\bfu_h)\right)_{\Omega}
          - \bscp{\Pi^{\ast}_h \bfeps(\bfu_h) \cdot \bfn,
          \bfu_h}_{\Gamma}
        \\
      \nonumber
        &\gtrsim 
        \bigl(
          \| \bfeps(\bfu_h) \|_{\Omega}
          + \bscp{h^{-1} \bfu_h, \bfu_h}_{\Gamma}^{1/2}
        \bigr)
        \| \bfeps(\bfu_h) \|_{\Omega}
        - \bscp{h^{-1} \bfu_h, \bfu_h}_{\Gamma}^{1/2} 
        \| \bfeps(\bfu_h) \|_{\Omega}
        \\
      \nonumber
        & \phantom{\gtrsim}
        +
            \left(\Pi^{\ast}_h \bfeps(\bfu_h) -
            \bfeps(\bfu_h),\,\bfeps(\bfu_h)\right)_{\Omega}
            - 
          \bscp{\Pi^{\ast}_h \bfeps(\bfu_h) \cdot \bfn,
          \bfu_h}_{\Gamma}
    \\
      \nonumber
        &\gtrsim 
    \| \bfu_h \|_{1,\Omega}
        \| \bfeps(\bfu_h) \|_{\Omega}
        - \bscp{h^{-1} \bfu_h, \bfu_h}_{\Gamma}^{1/2} 
        \| \bfeps(\bfu_h) \|_{\Omega}
        \\
      \nonumber
        & \phantom{\gtrsim}
        + \underbrace{
            \left(\Pi^{\ast}_h \bfeps(\bfu_h) -
            \bfeps(\bfu_h),\bfeps(\bfu_h)\right)_{\Omega}}_{\text{I}}
            - \underbrace{
          \bscp{\Pi^{\ast}_h \bfeps(\bfu_h) \cdot \bfn,
          \bfu_h}_{\Gamma}
    }_{\text{II}},
    \end{align}
    where we used the
    $L^2$-stability of the stabilized $L^2$-projection
    and variant~\eqref{eq:korn-III} of Korn's inequality
    in the last two steps. 
    We proceed by estimating the terms $I$ and $II$ separately.
    \\
    {Estimate (I)}: The use of the stabilized $L^2$-projection $\Pi^{\ast}_h$ allows to
    insert the Oswald interpolant of $\bfeps(\bfu_h)$ yielding
    \begin{align}
      \nonumber
        I &= 
      ( \Pi^{\ast}_h
      \bfeps(\bfu_h) - \bfeps(\bfu_h),\, \bfeps(\bfu_h) - \mcO_h\bfeps(\bfu_h))_{\Omega} 
      - s_{3}\left(\mcO_h \bfeps(\bfu_h), \Pi^{\ast}_h \bfeps(\bfu_h)\right)
      \nonumber
        \\ 
      \nonumber
        &\gtrsim -s_{1}(\bfu_h,\bfu_h)^{1/2} \| \Pi^{\ast}_h \bfeps( \bfu_h) -\bfeps(
        \bfu_h) \|_{\Omega}
      - s_{3}\left(\mcO_h \bfeps(\bfu_h), \Pi^{\ast}_h \bfeps(\bfu_h)\right)
      \nonumber
        \\
      \nonumber
        &\gtrsim 
      - s_{1}(\bfu_h,\bfu_h)^{1/2} \|\bfeps( \bfu_h) \|_{\Omega}
      - s_{3}\left(\mcO_h \bfeps(\bfu_h), \mcO_h \bfeps(\bfu_h)\right)^{1/2}
        \| \Pi^{\ast}_h \bfeps( \bfu_h) \|_{\Oast},
    \end{align}
    where we successively applied
    Lemma~\ref{lem:control-norm-with-ghost-penalties},
    the $L^2$-boundedness of $\Pi^{\ast}_{h}$, and finally a
    Cauchy-Schwarz inequality.
      By an argument similar to~\eqref{eq:b_h-inconsistency-estimate} in the previous
      Proposition~\ref{prop:mod-inf-sup-bh},
      we can show that
      \begin{align}
      \nonumber
        s_{3}\left(\mcO_h \bfeps(\bfu_h), \mcO_h \bfeps(\bfu_h)\right)
        \lesssim s_{1}(\bfu_h,\bfu_h)
      \end{align}
      and hence we arrive at
      \begin{align}
      \nonumber
        I 
        \gtrsim  
      - s_{1}(\bfu_h,\bfu_h)^{1/2} \|\bfeps( \bfu_h) \|_{\Omega}
      - s_{1}\left( \bfu_h, \bfu_h \right)^{1/2}
        \|  \Pi^{\ast}_h \bfeps( \bfu_h)  \|_{\Oast}.
      \end{align}
    \\
    { Estimate (II)}: Here we use inverse
    estimate~\eqref{eq:inverse-estimate-III} and the Nitsche penalty
    to control the boundary contribution:
    \begin{align}
      \nonumber
      \bscp{\Pi^{\ast}_h \bfeps(\bfu_h) \cdot \bfn, \bfu_h}_{\Gamma}
            &= 
            \bscp{ h^{1/2} \Pi^{\ast}_h \bfeps(\bfu_h) \cdot \bfn,
            h^{-1/2} \bfu_h}_{\Gamma}
            \\
            \nonumber
            &\lesssim \| \Pi^{\ast}_h  \bfeps(\bfu_h) \|_{\Oast}
            \bscp{h^{-1} \bfu_h, \bfu_h}^{1/2}_{\Gamma}
            \\
            \nonumber
            &\lesssim 
            \| \bfeps(\bfu_h) \|_{\Omega}
            \bscp{h^{-1} \bfu_h, \bfu_h}^{1/2}_{\Gamma}.
    \end{align}
    Collecting the estimates for $I$ and $II$ gives
    \begin{align}
      \nonumber
      a_h(\bfu_h, \bftau_h)
      &\gtrsim
      \| \bfu_h\|_{1,\Omega}
      \| \bfeps(\bfu_h) \|_{\Omega}
      - s_{1}(\bfu_h,\bfu_h)^{1/2} \|\bfeps( \bfu_h) \|_{\Omega}
      \\
      &\quad - s_{1}\left( \bfu_h, \bfu_h \right)^{1/2}
        \| \bftau_h \|_{\Oast}
      +  \| \bfeps(\bfu_h) \|_{\Omega}
      \bscp{h^{-1} \bfu_h, \bfu_h}^{1/2}_{\Gamma}.
      \label{eq:a_h-final-estimate}
    \end{align}
    Finally, we divide~\eqref{eq:a_h-final-estimate} by
    $\| \bfeps(\bfu_h) \|_{\Omega}$
    and recall $\| \bftau_h \|_{\Oast} \lesssim \| \bfeps(\bfu_h)
    \|_{\Omega}$ 
    to find that for some $c > 0$
  \begin{align}
      \nonumber
    \dfrac{a_h(\bftau_h, \, \bfu_h)}{\| \bftau_h \|_{\Oast}}
 \gtrsim 
    \dfrac{a_h(\bftau_h, \, \bfu_h)}{\| \bfeps(\bfu_h) \|_{\Omega}}
 \gtrsim
    \| \bfu_h \|_{1,\Omega} -  c ( s_{1}(\bfu_h,\bfu_h)^{1/2} +
    \bscp{h^{-1} \bfu_h, \bfu_h}^{1/2}_{\Gamma})
    \label{eq:mod-inf-sup-ah}
  \end{align}
    if  $a_h(\bftau_h, \, \bfu_h) \geqslant 0$, otherwise we proceed
    as in the previous proof.
\end{proof}

Combining the modified inf-sup conditions~\eqref{eq:mod-inf-sup-bh}
and~\eqref{eq:mod-inf-sup-ah} enable us to prove an inf-sup
condition for the discrete variational
problem~\eqref{prob:tfstokes-cutfem} with respect to the total
approximation space $\V_h$.

\begin{theorem}
  It holds
  \begin{equation}
    \sup_{V_h \in \V_h\setminus \{0\}} 
    \dfrac{A_h(U_h,V_h) + S_h(U_h,V_h)}{\tn V_h \tn_h}
    \gtrsim
    \tn U_h \tn_h, \quad \foralls U_h \in \V_h.
    \label{eq:inf-sup-Ah-Sh}
  \end{equation}
  \label{thm:inf-sup-Ah-Sh}
\end{theorem}
\begin{proof}
  Given $U_h = (\bfsigma_h, \bfu_h, p_h)$, we construct a proper test
  function $V_h$ in four steps. 
  \\
  {\bf I}: Choosing $V^1_h = U_h$, we obtain
  \begin{align}
    \label{eq:inf-sup-A1-term}
    A_h(U_h,V_h^1)  + S_h(U_h, V_h^1)
    = \dfrac{1}{2 \eta} \| \bfsigma_h \|_{\Omega}^2
    + 2 \eta \gamma_b \|h^{-1/2}\bfu_h \|_{\Gamma}^2
    + S_h(U_h, U_h).
  \end{align}
  \\
  {\bf II}: Now we choose $V_h^2 = (0, \bfv_h^p,0)$ , where $\bfv_h^p$
  attains the supremum in~\eqref{eq:mod-inf-sup-bh} for given $p_h$ and is rescaled
  such that $\eta \| \bfv_h^p \|_{1, h}^2 = \tfrac{1}{\eta} \| p_h
  \|_{\Omega}^2$. Then writing
    $A_2 = A_h(U_h,V_h^2) + S_h(U_h,V_h^2)$ and applying
    Cauchy-Schwarz and the modified inf-sup
    condition~\eqref{eq:mod-inf-sup-bh}, we obtain
  \begin{align}
    A_2 
    &= a_h(\sigma_h, \bfv_h^p) + b_h(p_h, \bfv_h^p) + s_u(\bfu_h,\bfv_h^p)  + \frac{\gamma_b \eta}{h} 
    \bscp{
    u_h,\,v_h^p}_{\Gamma} 
    \nonumber
    \\
    &\gtrsim
    - \| \bfsigma_h \|_{\Omega} \| \bfeps(\bfv_h^p) \|_{\Omega} 
    - \| h^{1/2} \bfsigma_h \|_{\Gamma} \| h^{-1/2} \bfv_h^p \cdot n \|_{\Gamma}
    \nonumber
    \\
    &\quad 
    + \| p_h \|_{\Omega}  \| \bfv_h^p \|_{1,h} 
    - s_p(p_h, p_h)^{1/2} \| \bfv_h^p \|_{1,h} 
    - \dfrac{\eta \gamma_b}{h} \|\bfu_h \|_{\Gamma} \| \bfv_h \|_{\Gamma}
    + s_u(\bfu_h, \bfv_h^p)
    \nonumber
    \\
    &\gtrsim
    - \dfrac{\delta^{-1}}{\eta} \| \bfsigma_h \|_{\Omega}^2
    -        \delta \eta \| \bfeps(\bfv_h^p) \|_{\Omega}^2 
    - \dfrac{\delta^{-1}}{\eta} \| h^{1/2} \bfsigma_h \|_{\Gamma}^2
    -        \delta \eta \| h^{-1/2} \bfv_h^p \cdot n \|_{\Gamma}^2
    \nonumber
    \\
    &\quad 
    + \dfrac{1}{\eta}\| p_h \|_{\Omega}^2 
    - \dfrac{\delta^{-1}}{\eta} s_p(p_h, p_h) - \delta \eta  \| \bfv_h^p
    \|_{1,h}^2
    \nonumber
  - \delta^{-1} \eta \gamma_b \|h^{-1/2} \bfu_h \|_{\Gamma}^2
  - \delta \eta \gamma_b \|h^{-1/2} \bfv_h^p \|_{\Gamma}^2
    \nonumber
    \\
    &\quad 
    - \dfrac{\delta^{-1}}{\eta} s_u(\bfu_h, \bfu_h)
    - \delta \eta s_u(\bfv_h^p, \bfv_h^p),
  \end{align} 
  where a $\delta$-weighted arithmetic-geometric inequality was used
  in the last step.
  Due to the scaling choice $\eta \| \bfv_h^p \|_{1,h}^2 = \tfrac{1}{\eta} \| p_h
  \|_{\Omega}^2$, all $\bfv_h^p$-related terms can be absorbed into $\| p_h \|_{\Omega}$
  by choosing $\delta$ small enough. 
  If we then combine an inverse estimate and
  Lemma~\ref{lem:control-norm-with-ghost-penalties}
  to estimate the boundary term $\| h^{1/2} \bfsigma_h \|_{\Gamma}$ by
  \begin{align}
    \dfrac{1}{\eta}\| h^{1/2} \bfsigma_h \|_{\Gamma}^2 
    \lesssim 
    \dfrac{1}{ \eta}
    \| 
    \bfsigma_h
    \|_{\Oast}^2
    \lesssim  
    \dfrac{1}{ \eta}
    \| \bfsigma_h \|_{\Omega}^2 + 
    s_{\bfsigma}(\bfsigma_h,
    \bfsigma_h),
  \end{align}
  we see that there exists
  constants $C_p$ such that 
  \begin{align}
    A_2 
    &\gtrsim
    \dfrac{1}{\eta}\| p_h \|_{\Oast}^2 
    - C_p 
    \left(
    \dfrac{1}{\eta} \| \bfsigma_h \|_{\Omega}^2
    + \eta \gamma_b \|h^{-1/2}\bfu_h \|_{\Gamma}^2
    + S_h(U_h, U_h)
    \right),
    \label{eq:inf-sup-A2-term}
  \end{align}
  where we also applied Lemma~\ref{lem:control-norm-with-ghost-penalties}
  to pass from  $\| p_h \|_{\Omega}$ to  $\| p_h \|_{\Oast}$ 
  via the term $s_p(p_h, p_h)$.

  {\bf III}: Next, we pick $V_h^3 = (\bftau_h^u, 0 ,0)$, where
  $\bftau^u_h$ attains the supremum in~\eqref{eq:mod-inf-sup-ah} for the given $\bfu_h$
  and is rescaled such that $ \dfrac{1}{\eta} \| \bftau_h^u
  \|_{\Oast}^2 =  \eta \| \bfu_h \|_{1,\Omega}^2$.
  Introducing $A_3 =  A_h(U_h,V_h^3) +
  S_h(U_h,V_h^3)$,
  we can bound $A_3$ along the same lines as in the previous step:
  \begin{align}
    A_3 
    &= \dfrac{1}{2 \eta} (\bfsigma_h, \bftau_h^p)_{\Omega}
    - a_h(\bftau_h^p, \bfu_h)_{\Omega}
    + s_{\bfsigma}(\bfsigma_h, \bftau_h^u)
    \nonumber
    \\
    & \gtrsim
    - \dfrac{\delta^{-1}}{\eta} \| \bfsigma_h \|_{\Omega}^2
    - \dfrac{\delta}{ \eta} \| \bftau_h^u \|_{\Omega}^2
    + \eta \| \bfu_h \|_{1,\Omega}^2 
    - \delta^{-1} s_u(\bfu_h, \bfu_h)
    - \delta^{-1} \eta \| h^{-1/2} \bfu_h \|_{\Gamma}^2
    \nonumber
    \\
    &\quad 
    - \delta^{-1} s_{\sigma}(\bfsigma_h, \bfsigma_h)
    - \delta s_{\sigma}(\bftau_h^u, \bftau_h^u)
    \nonumber
    \\
    & \gtrsim
    \eta \| \bfu_h \|_{1,\Oast}^2 -
     C_u 
     \left(
    \dfrac{1}{\eta} \| \bfsigma_h \|_{\Omega}^2
    + \eta \gamma_b \|h^{-1/2}\bfu_h \|_{\Gamma}^2
    + S_h(U_h, U_h)
     \right).
    \label{eq:inf-sup-A3-term}
  \end{align}
  {\bf IV}: Finally, we define $V_h = V_h^1 + \alpha V_h^2 + \beta
  V_h^3$. 
  Combining the
  estimates~\eqref{eq:inf-sup-A1-term}, \eqref{eq:inf-sup-A2-term}
  and \eqref{eq:inf-sup-A3-term}, we observe that by choosing $\alpha$
  and $\beta$ small enough, it holds that
  \begin{align}
    A_h(U_h,V_h) + S_h(U_h,V_h)
    \gtrsim \tn U_h \tn_h^2.
    \nonumber
  \end{align}
  To conclude the proof, we note that by our choices of $V_h^i$, $i =
  1,2,3$,
  \begin{align*}
    \tn V_h^1 \tn_h^2 &= \tn U_h \tn_h^2,
    \\
    \tn V_h^2 \tn_h^2 &= 2 \eta \| \epsilon(v_h^p) \|_{\Omega}^2 + 2 \eta \gamma_b
    \|h^{-1/2} v_h^p \|_{\Gamma}^2 + s_u(v_h^p, v_h^p)
    \lesssim 
    \eta \| v_h^p \|_{1,h}^2 = \dfrac{1}{\eta}
    \| p_h \|^2_{\Omega},
    \\
    \tn V_h^3 \tn_h^2 &= \dfrac{1}{2 \eta} \tn \tau_h^u \tn_{\Omega}^2 +
    s_{\sigma}(\tau_h^u, \tau_h^u)
    \lesssim \dfrac{1}{\eta} \| \tau_h^u \|_{\Oast}
    = \eta \| u_h \|^2_{1,\Omega},
  \end{align*}
  and thus $\tn V_h \tn_h \lesssim \tn U_h \tn_h$
  which proves the desired estimate.
\end{proof}

\section{A priori estimates}
\label{sec:apriori-analysis}
In this section, we state and prove the a priori estimate for the error in
the discrete solution, defined by problem~\eqref{prob:tfstokes-cutfem}.
Before we present the main result, we state two lemmas which quantify
the effect of the consistency error introduced by the
stabilization term $S_h$. The first lemma ensures that 
a weakened form of the Galerkin orthogonality holds:
\begin{proposition}
  Let $(\sigma_h, u_h, p_h) \in \V_h$ be the finite element
  approximation defined by~\eqref{prob:tfstokes-cutfem}
  and assume that
  the weak solution  
  $(\bfsigma, \bfu, p)$ 
  of the three field Stokes problem~\eqref{eq:tfstokes-weak}
  is in  $[H^1(\Omega)]^{d\times d} 
  \times [H^2_0(\Omega)]^d \times H^1(\Omega)$.
  Then
  \begin{align}
    A_h(U - U_h, V_h) =  S(U_h, V_h).
  \end{align}
\end{proposition}
\begin{proof}
  The proof follows immediately from the definition of the weak
  variational problem~\eqref{eq:tfstokes-weak} and the easily verified fact
  that the continuous solution $U$ satisfies $A_h(U,V_h) = L_h(V_h)$.
\end{proof}

The second lemma ensures that
the consistency error does not make the convergence rate deteriorate. 
\begin{proposition}
  Suppose that $U = (\bfsigma, \bfu, p) \in [H^1(\Omega)]^{d \times d} \times
  [H^2(\Omega)]^d \times H^1(\Omega)$, then it holds that
  \begin{align}
    \label{eq:consistency-error}
    | S_h(\mcC_h^* U, V_h) | \lesssim h
    \left(
    \eta^{1/2}\|\bfu\|_{2,\Omega} 
    +\frac{1}{\eta^{1/2}}\|p\|_{1,\Omega}
    +\frac{1}{\eta^{1/2}}\|\bfsigma\|_{1,\Omega}
    \right)
    \tn V_h \tn_h.
  \end{align}
\end{proposition}
\begin{proof}
  By definition, 
  \begin{align*}
    S_h(\mcC_h^* U, V_h) 
    = s_{\sigma}(\mcC_h^* \bfsigma, \bftau_h) 
    + s_u(\mcC_h^* u, v_h) 
    + s_p(\mcC_h^* p, q_h).
  \end{align*}
  We start with the velocity related terms.  Since we assume that 
   $u \in H^2(\Omega) \cap H^1_0(\Omega)$,
  we have $s_u(\bfu^{\ast},\bfv_h) = 0$ for its extension $u^{\ast} =
  Eu$ to $\Oast$. Exploiting this fact together with the trace
  inequality~\eqref{eq:trace-inequality}, the inverse
  estimate~\eqref{eq:inverse-estimate-III}, the interpolation estimate~\eqref{eq:interpest1}
  and the stability of the interpolation operator $\mcC_h^*$, we
  might estimate the velocity part of the consistency error as
  follows:
  \begin{align*}
    | s_u(\mcC_h^* \bfu, \bfv_h) |
    &=  
    | s_u(\mcC_h^* \bfu - \bfu^{\ast}, \bfv_h) |
    \lesssim
    \eta \sum_{F\in \mcF_i} 
    h^{1/2}\| \nablan (\mcC_h^* \bfu - \bfu^{\ast}) \|_{F}
    \,
    h^{1/2}\| \nablan \bfv_h \|_{F}
    \\
    & \lesssim
    \eta^{1/2}  \left(\sum_{T\in \mesh} 
    \left(h
    \| \nabla (\mcC_h^* \bfu - \bfu^{\ast}) \|^2_{T}
    + \|  (\mcC_h^* \bfu - \bfu^{\ast}) \|^2_{T}
  \right) \right)^{\frac12}
    \eta^{1/2} \| \nabla \bfv_h \|_{\Oast}
    \\
    &\lesssim h \eta^{1/2} \| \bfu^{\ast} \|_{2,\Oast} \| \bfv_h \|_{1,\Oast}
    \lesssim  h \eta^{1/2} \| \bfu \|_{2,\Omega} \tn V_h \tn_h.
  \end{align*}

  For the pressure, applying the inverse
  inequality~\eqref{eq:inverse-estimate-II} and the boundedness of the
  interpolation operator~\eqref{eq:interpest1} gives
  \begin{align*}
    | s_p(\mcC_h^* p, q_h ) |
    \lesssim
    h \eta^{-1} \| \nabla \mcC_h^* p \|_{\Oast}  
    h \| \nabla q_h \|_{\Oast}
    \lesssim
    h \eta^{-1} \| p \|_{1,\Omega}  
    \| q_h \|_{\Oast}
    \lesssim
    h \eta^{-1/2} \| p \|_{1,\Omega}  
    \tn V_h \tn_h.
  \end{align*}
  Finally, we observe that the consistency error in $\bfsigma_h$
  might be bounded by applying the same steps as for the pressure
  related terms.
\end{proof}

We are now in the position to state our main result.
\begin{theorem}
  Let $U = (\bfsigma,\bfu,p) \in [H^1(\Omega)]^{d \times d} \times
  [H^2(\Omega)]^d \times H^1(\Omega)$ be the solution of the
  three field Stokes problem~\eqref{eq:tfstokes-strong} and let 
  $U_h = (\bfsigma_h, \bfu_h, p_h)$ be the solution to the
  discrete problem~\eqref{prob:tfstokes-cutfem}.
  Then the following error estimate holds:
  \[
  \tnorm{U-U_h}\lesssim h 
  \left(
    \eta^{1/2}\|\bfu\|_{2,\Omega} +
    \frac{1}{\eta^{1/2}}\|p\|_{1,\Omega} + \frac{1}{\eta^{1/2}}
    \|\bfsigma\|_{1,\Omega}
  \right),
  \]
  where the hidden constant is independent of how the boundary cuts the mesh.
\end{theorem}
\begin{proof}
  Using the triangle inequality $\tn U - U_h \tn \lesssim \tn U -
  \mcC_h^* U \tn + \tn U_h - \mcC_h^* U \tn_h$ and the standard
  interpolation estimates~\eqref{eq:interpest0}, we can see that
  the error $\tn U - \mcC_h^* U \tn$ satisfies the
  desired estimate. By inf-sup condition~\eqref{eq:inf-sup-Ah-Sh}
  and the weak Galerkin orthogonality, there exists a $V_h$ such that
  \begin{align}
    \tn U_h -  \mcC_h^* U \tn_h
    &\lesssim 
    \dfrac{A_h(U_h -  \mcC_h^* U, V_h) + S_h(U_h - \mcC_h^* U, V_h) }{\tn V_h \tn_h}
    \\
    &= 
    \dfrac{A_h(U -  \mcC_h^* U, V_h) - S_h(\mcC_h^* U, V_h) }{\tn V_h \tn_h}
    = A + S.
    \label{}
  \end{align}
  Recalling the bound for the consistency
  error~\eqref{eq:consistency-error}, it suffices to estimate
  \begin{align}
    \nonumber
    A &= 
    \dfrac{1}{2 \eta} (\bfsigma - \mcC_h^* \bfsigma, \bftau_h)_{\Omega}
    + 2 \eta \gamma_b \bscp{h^{-1} (\bfu -  \mcC_h^* \bfu),
    \bfv_h}_{\Gamma}
    + a_h(\bfsigma - \mcC_h^* \bfsigma, \bfv_h)
    - a_h(\bftau_h, \bfu - \mcC_h^* \bfu) 
    \\
    \nonumber
    &\quad
    + b_h(p - \mcC_h^* p, \bfv_h)
    - b_h(q_h, \bfu - \mcC_h^* \bfu). 
  \end{align}
  For the first term, we simply have 
  \begin{align*}
   |\dfrac{1}{2 \eta} (\bfsigma - \mcC_h^* \bfsigma,
   \bftau_h)_{\Omega}|
   &\lesssim
   \dfrac{h}{\eta^{1/2}} \| \bfsigma \|_{1,\Omega} \tn V_h \tn_h,
  \end{align*}
  while for the second term, combining
  the trace inequality~\eqref{eq:trace-inequality}
  with the interpolation estimate~\eqref{eq:interpest0-ast}
  yields
  \begin{align*}
    2 \eta \gamma_b \bscp{h^{-1} (\bfu -  \mcC_h^* \bfu),
    \bfv_h}_{\Gamma}
    \lesssim (2 \eta \gamma_b)^{1/2} h \|\bfu \|_{2,\Omega} \tn V_h
    \tn_h.
  \end{align*}
  Next, the third term can be estimated by
  \begin{align*}
   |a_h(\bfsigma - \mcC_h^* \bfsigma, \bfv_h)|
   &\lesssim
   \dfrac{1}{\eta^{1/2}} 
   ( \| \bfsigma - \mcC_h^* \bfsigma \|_{\Omega}
   + \|h^{1/2}(\bfsigma  - \mcC_h^* \bfsigma)\|_{\Gamma}
   )
   \cdot \eta^{1/2}
   (\| \bfeps(\bfv_h) \|_{\Omega}
    + \| h^{-1/2} \bfv_h \|_{\Gamma}
    )
    \\
    &\lesssim
    \dfrac{h}{\eta^{1/2}} \| \bfsigma \|_{1,\Omega}
    \tn V_h \tn_h.
  \end{align*}
  Similarly, the fourth term can be bounded
  \begin{align}
    | a_h(\bftau_h, \bfu - \mcC_h^* \bfu ) |
    & \lesssim
    \dfrac{1}{\eta^{1/2}}
    (
    \| \bftau_h \|_{\Omega} + \| h^{1/2} \bftau_h \|_{\Gamma}
    )
    \; \eta^{1/2}
    (
    \| \bfeps(\bfu) - \mcC_h^*\bfu \|_{\Omega}
    + \| h^{-1/2}(\bfu - \mcC_h^*\bfu) \|_{\Gamma}
    )
    \nonumber
    \\
    &\lesssim
    h \tn V_h \tn_h \eta^{1/2} \| \bfu \|_{2,\Omega}. 
    \label{eq:apriori-step-3}
  \end{align}
  Here, we estimated the boundary term in~\eqref{eq:apriori-step-3}
  by successively applying
  the trace inequality~\eqref{eq:trace-inequality}, standard
  interpolation estimates and the boundedness of the extension
  operator $E: H^2(\Omega) \to H^2(\Oast)$,
  cf.~\eqref{eq:extension-operator-boundedness}, which yields
  \begin{align*}
    \| h^{-1/2}(\bfu - \mcC_h^*\bfu) \|_{\Gamma}
    \lesssim
    h^{-1} \| \bfu^{\ast} - \mcC_h^*\bfu \|_{\Oast}
    + 
    h \| \bfu^{\ast} - \mcC_h^* \bfu \|_{1,\Oast}
    \lesssim h \| \bfu \|_{2,\Omega}.
  \end{align*}
  The estimates for the remaining terms involving $b_h(\cdot,\cdot)$ are
  completely analogous, which concludes the proof.
\end{proof}
\begin{remark}
To reduce the system matrix stencil, one may use the
element-based penalty terms
\begin{align}
s_p(p_h,q_h) &= \frac{\gamma_p}{2 \eta} \sum_{T \in \mesh} h^2(\nabla p_h, \nabla q_h)_{T}, \\ 
  s_{\sigma}(\sigma_h,\tau_h) &= \frac{\gamma_{\sigma}}{2 \eta} \sum_{T \in \mesh} h^2(\nabla \sigma_h, \nabla \tau_h)_{T} 
\end{align}
for the pressure and stress, instead of the face-based penalty
terms~\eqref{eq:stabilization-operators},
\eqref{eq:stabilization-operator-p} over gradient jumps, if linear
finite element spaces are chosen for velocity, pressure and stress. 
Note that both the face and element-based penalty terms are weakly
consistent for $\bfsigma \in [H^1(\Omega)]^{d \times d}$ and $p \in
H^1(\Omega)$.
However, the face-based penalty term~\eqref{eq:stabilization-operator-u}
for $u \in [H^2(\Omega)]^d$ is strongly consistent and thus
strictly necessary as the analogous element-based penalty
term leads to a consistency error which deteriorates the overall
convergence order.
\end{remark}

\section{Numerical results}
\label{sec: numerical results}
In this section, we will demonstrate that the theoretical estimates of
Section~\ref{sec:stability} and Section~\ref{sec:apriori-analysis}
hold. In particular, we will show that the finite element solution of
velocity, pressure and extra-stress tensor converge with optimal
order to a $\sin$-$\cos$ reference solution of the three field Stokes system
and we will demonstrate that the ghost penalties yield independence of
the quality of the solution on the boundary location. All
numerical simulations have been performed using our software package
libCutFEM which will be made available soon at \url{http://www.cutfem.org}.
LibCutFEM is an open source library 
which extends the finite element library DOLFIN~\cite{LoggWells2010a}
and the FEniCS framework~\cite{LoggMardalEtAl2011} for automated
computing of finite element variational problems with cut finite
element capabilities.  The inner workings of libCutFEM are described
as part of the review article \cite{BurmanClausHansboEtAl2014}.

\subsection{Convergence study for reference solution}
To evaluate the accuracy of our scheme, we investigate the rate of convergence of the numerical solution to the following reference solution
\begin{align}
\bfu_{ex} = \left[\begin{matrix} -\sin(\pi y)\cos(\pi x) \\ \sin(\pi x)\cos(\pi y) \end{matrix}\right], \nonumber \\
p_{ex} = -2\eta\cos(\pi x)\sin(\pi y), \nonumber \\
\bfsigma_{ex} = \left[\begin{matrix}2.0 \pi \eta \sin{\left (\pi x \right )} \sin{\left (\pi y \right )} & 0\\0 & - 2.0 \pi \eta \sin{\left (\pi x \right )} \sin{\left (\pi y \right )}\end{matrix}\right], \nonumber \\
\bff = \left[\begin{matrix}2 \pi \eta \sin{\left (\pi x \right )} \sin{\left (\pi y \right )} - 2 \pi^{2} \eta \sin{\left (\pi y \right )} \cos{\left (\pi x \right )}\\2 \pi^{2} \eta \sin{\left (\pi x \right )} \cos{\left (\pi y \right )} - 2 \pi \eta \cos{\left (\pi x \right )} \cos{\left (\pi y \right )}\end{matrix}\right] 
\label{equ: 2d sincos solution}
\end{align}
of the three field Stokes system. Here, we choose $\gamma_u= 0.01$, $\gamma_p= 0.1$, $\gamma_{\bfsigma}=0.1$, $\gamma_b=15.0$ and $\eta= 0.5$ and compute the velocity, pressure and extra-stress in a unit circle embedded in a fixed background mesh. We set 
$
\bfu = \bfu_{ex} \mbox{ at } \partial \Omega.
$
For the velocity and the extra-stress tensor, the sum of the error of
the components is evaluated. \\
The rate of convergence for the
$L^2$-error, $|| U_h - U_{ex}||_0$, and for the $H^1$ error of $|| u_h - u_{ex}||_1$ are displayed in Figure~\ref{fig:
convergence rate 2d}. We obtain a convergence order of 1.05 for the
velocity in the $H^1$-norm, which is what we expect from our error analysis.  In the $L^2$ norm, the velocity converges with order 2.18. We obtain a convergence order of 1.77 and 1.99 for the extra-stress and
pressure which is better than expected.
However, this can be
explained by the smoothness of the solution.
\begin{figure}[htb!]
\centering
\subfloat[Convergence rate.]{\includegraphics[width=.5\textwidth]{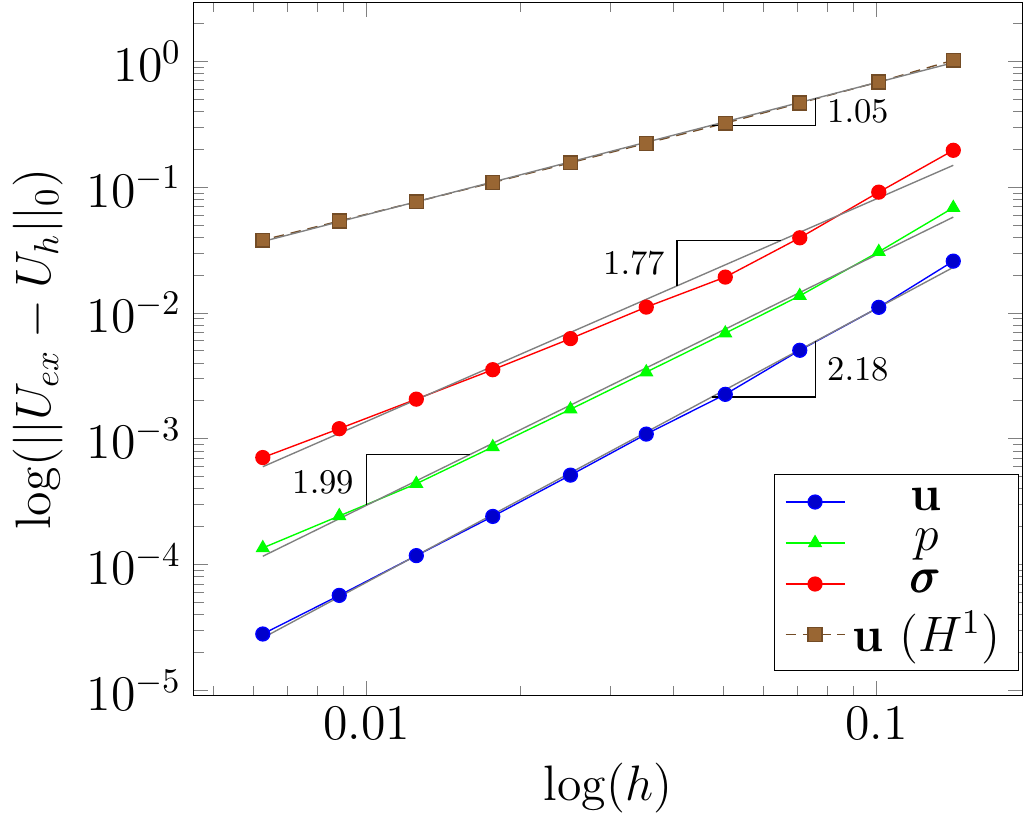}}  
\subfloat[Unit circle domain.]{\includegraphics[width=.5\textwidth]{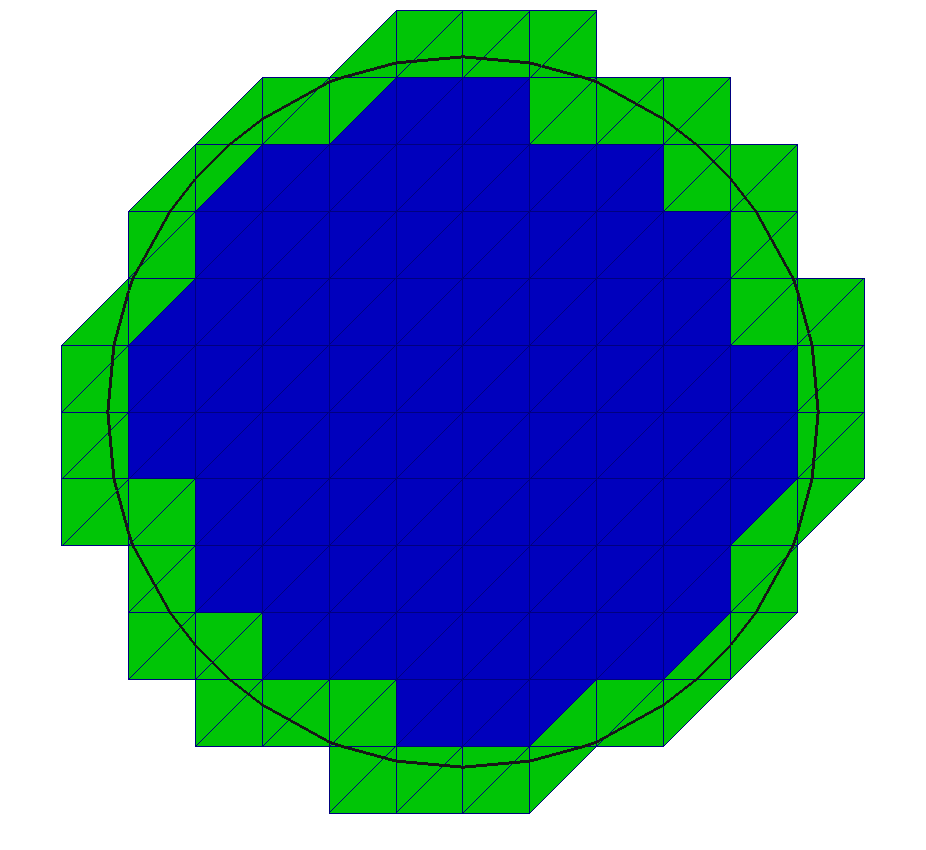}}  
\caption{Rate of convergence for two-dimensional sin-cos reference solution \eqref{equ: 2d sincos solution} in a unit circle domain.}
\label{fig: convergence rate 2d}
\end{figure}
\subsection{Stabilizing effect of ghost penalty terms}
In this section, we investigate the quality of the solution with respect to how the interface cuts the mesh. A boundary location in which only very small parts of the elements in the interface zone are covered by the physical domain can lead to an ill-conditioned system matrix and to an unbounded Nitsche boundary penalty parameter $\gamma_b$. To demonstrate that this ill-conditioning and the unboundedness of the penalty parameter can be alleviated using ghost penalties, we investigate the quality of the solution in terms of a boundary location parameter $0<\epsilon<1$. This parameter $\epsilon$ indicates the relative height of thin fluid stripes in a boundary cut parallel to an element edge, see Figure~\ref{subfig: sliverness def}. We call this type of cut configuration, the sliver case and $\epsilon$ indicates the sliver size. 

\subsubsection{Dependence of the quality of the solution on the sliver size}
Consider the reference solution \eqref{equ: 2d sincos solution} in a square domain $\Omega=[-1,1]^2$ embedded in a dilated background mesh of size 
\begin{equation}
\Oast=[-1-l,1+l] \mbox{ with } l=\frac{2(1-\epsilon)} {N-2(1-\epsilon)},
\end{equation}
where $N$ is the number of elements in the $x$ and in the $y$
direction. Figure~\ref{subfig: sliver test mesh} shows the
approximated interface location of the quadratic domain in dilated
background meshes for $\epsilon=0.5$ and $\epsilon=0.1$. We
investigate the effect of the ghost penalty parameter
$\gamma_{\bfsigma}$ on the quality of the sin-cos reference solution
for $\epsilon=\left\{0.5,0.1,0.02,0.004\right\}$. Throughout this
section, we set $\gamma_b=15.0$, $\gamma_{\bfu}=0.1$ and
$\gamma_p=0.1$. Figure~\ref{subfig: conv sliver} shows that for
$\epsilon=0.02$ and $\gamma_{\sigma}=0.1$, the extra-stress, the
velocity and the pressure converge with the optimal order of
convergence as predicted by the analysis in Section
\ref{sec:apriori-analysis}. Setting the ghost penalty parameter to
$\gamma_{\bfsigma}=0.0$ for $\epsilon=0.02$ causes an upward shift of
the error for the extra-stress tensor as shown in Figure~\ref{subfig:
conv sliver2}. Figure~\ref{subfig: stress wo gammas} shows this
increase of the error in the extra-stress tensor with decreasing
sliver size $\epsilon$ for the unstabilized extra-stress tensor
variable. Using the ghost penalty stabilization
($\gamma_{\bfsigma}=0.1$), this increase in error can be alleviated
and the solution becomes independent of the boundary location (see
Figure~\ref{subfig: stress w gammas}).  The cause for the error
increase for the unstabilized extra-stress variable can be observed in
Figure~\ref{fig: gammas ghost stab} for the extra-stress tensor
component $\sigma_{xx}$. Without the ghost penalty stabilization, we
have huge spikes appearing at the corner of the domain in the solution
and the solution shows large oscillations along the boundary. Even
though these spikes and oscillations decrease with mesh refinement the
solution of the extra-stress tensor component is polluted by the poor
solution in the boundary region. Setting $\gamma_{\bfsigma}=0.1$
alleviates this problem and the solution does not undergo any large
spikes or oscillations in the boundary region.  

\begin{figure}[htb]
\subfloat[ \label{subfig: sliverness def}]{\includegraphics[width=.3\textwidth]{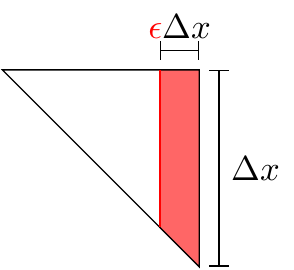}}
\subfloat[\label{subfig: sliver test mesh} ]{\includegraphics[width=.7\textwidth]{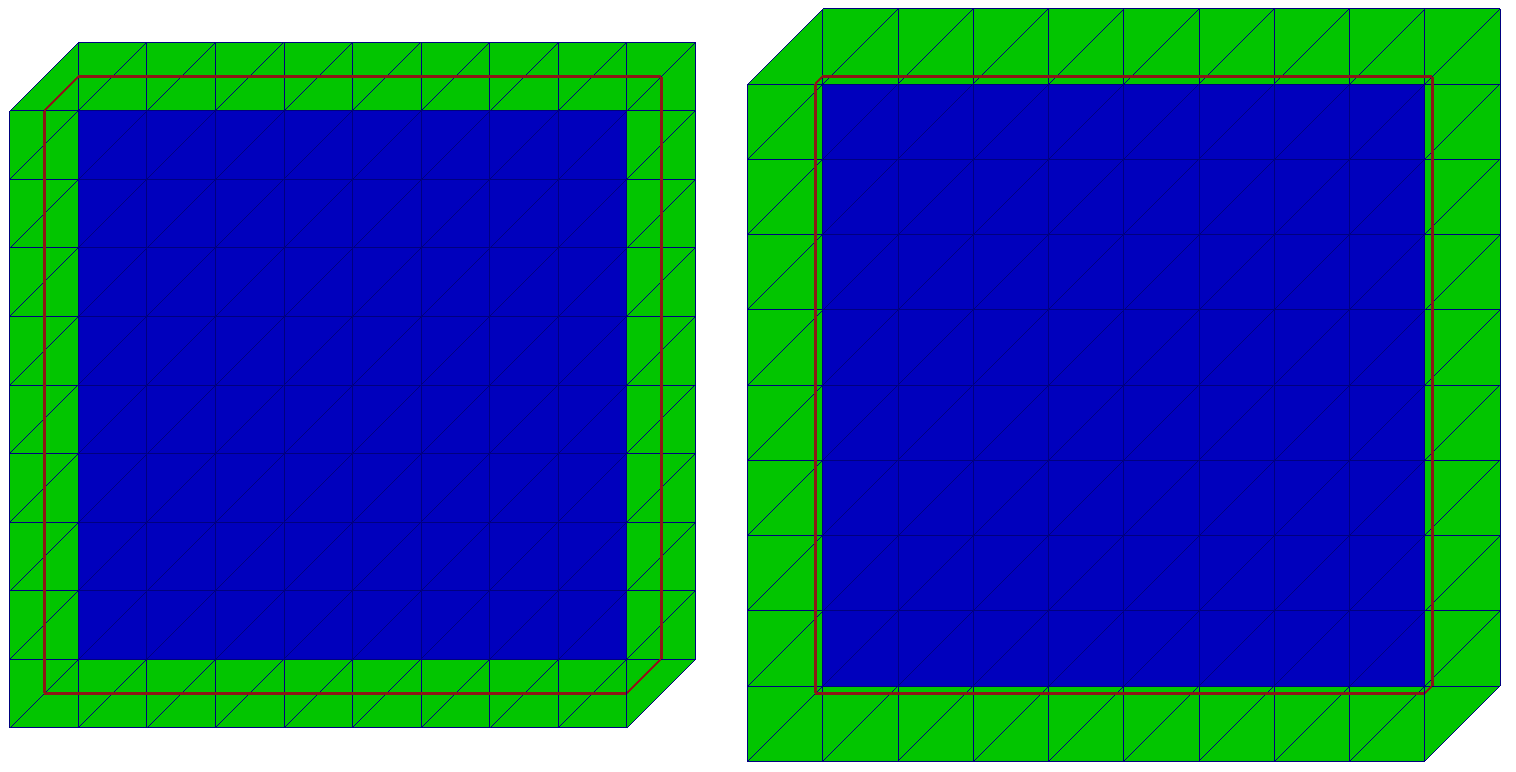}}
\caption{Schematics of \protect\subref{subfig: sliverness def} the definition of the sliver size parameter $\epsilon$ and \protect\subref{subfig: sliver test mesh} the approximated interface location of   $\Omega=\left[-1,1\right]$ in dilated background meshes for $\epsilon=0.5$ and $\epsilon=0.1$.}
\label{fig: sliver test schematics}
\end{figure}

\begin{figure}[htb]
\centering
\subfloat[ \label{subfig: conv sliver}$\gamma_{\bfsigma}=0.1$, $\epsilon=0.02$.]{\includegraphics[width=.45\textwidth]{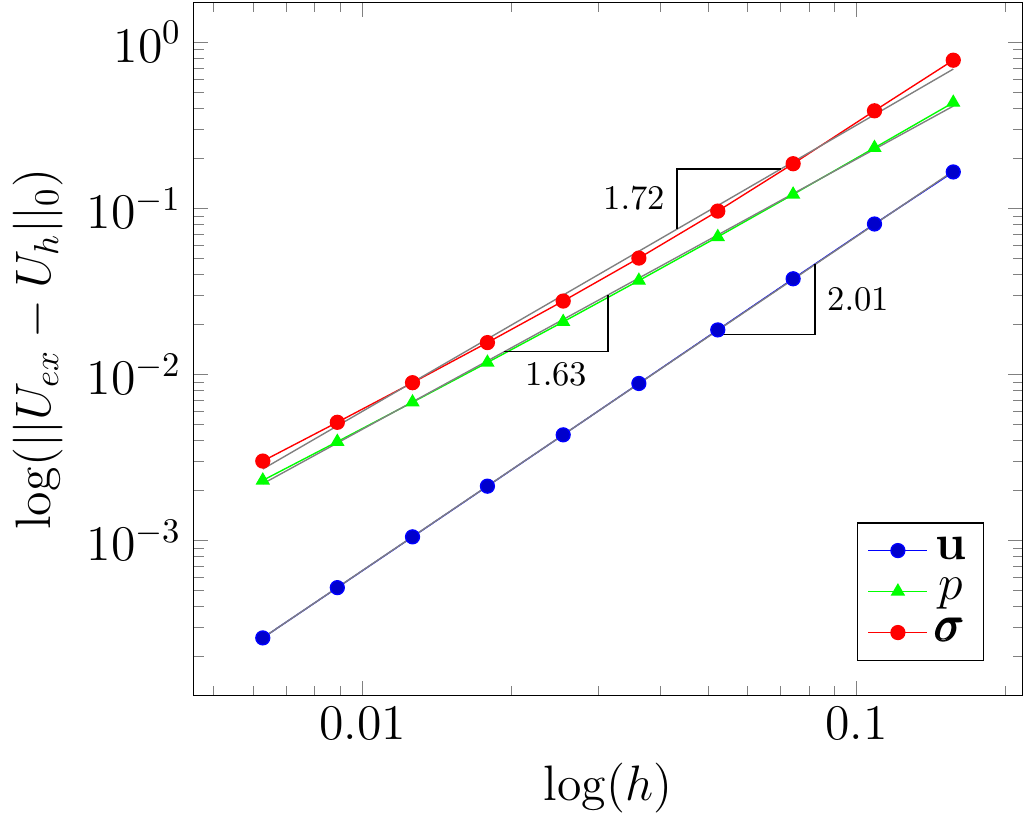}} 
\subfloat[ \label{subfig: conv sliver2}$\gamma_{\bfsigma}=0.0$, $\epsilon=0.02$.]{\includegraphics[width=.45\textwidth]{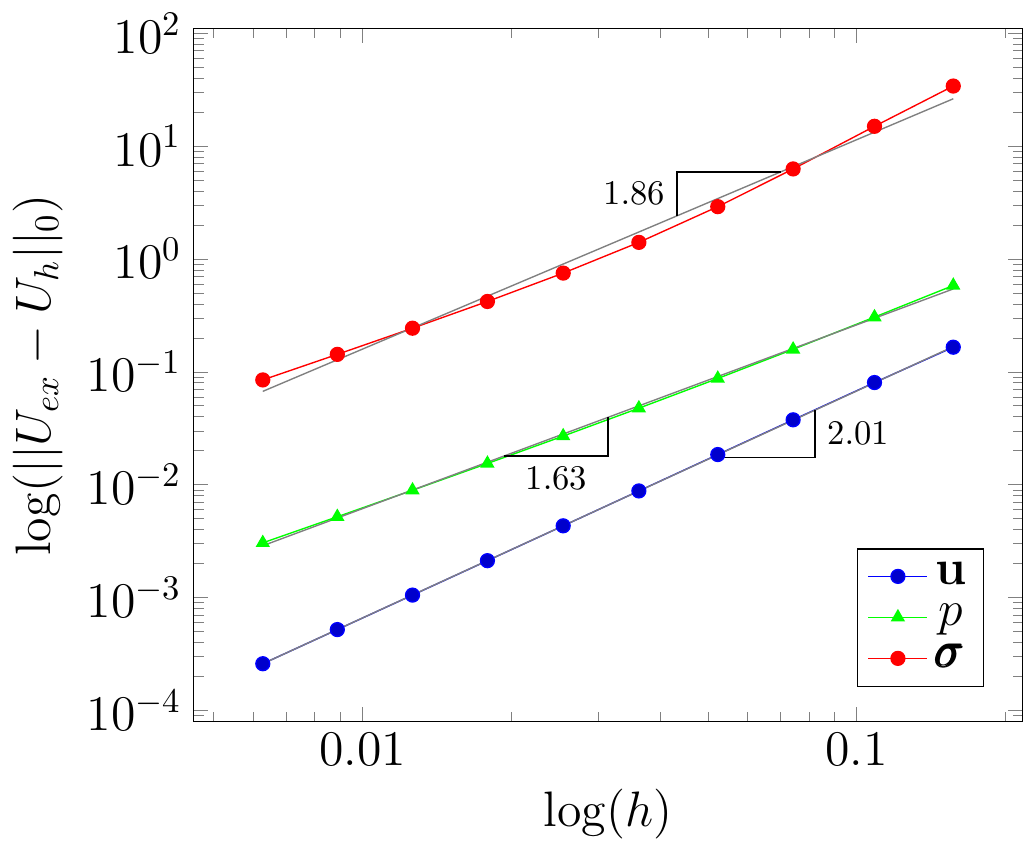}} \\
\subfloat[\label{subfig: stress w gammas} $\gamma_{\bfsigma}=0.1$.]{\includegraphics[width=.45\textwidth]{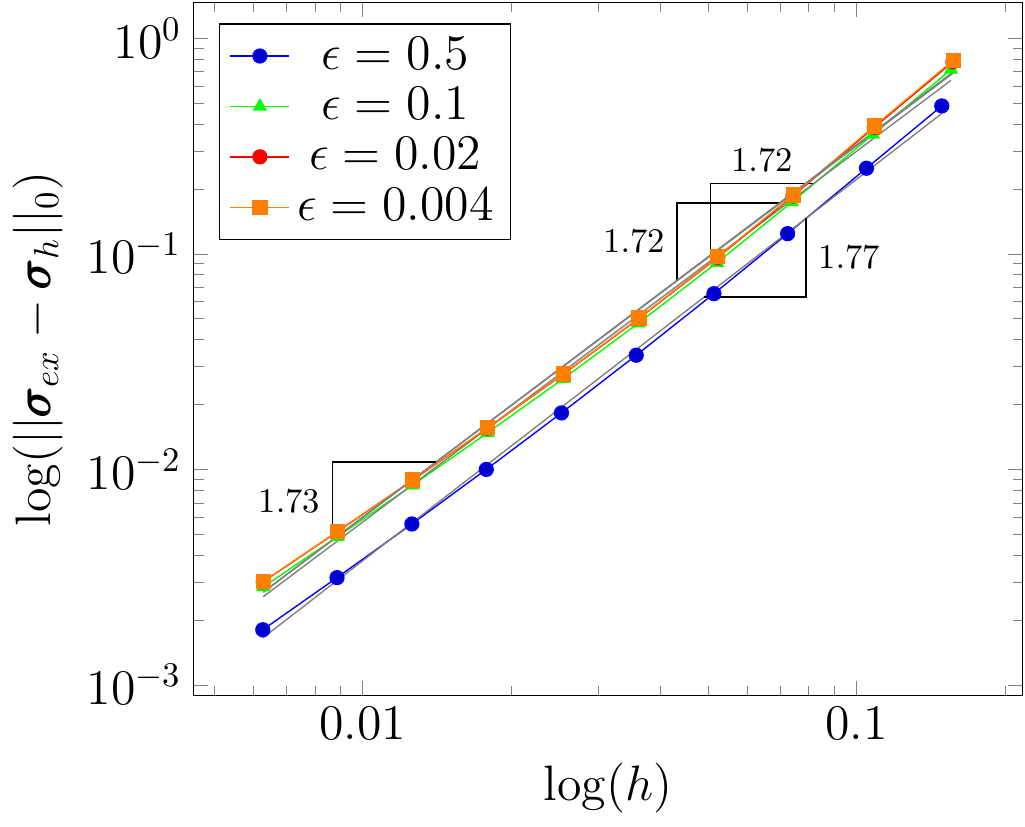}}
\subfloat[\label{subfig: stress wo gammas} $\gamma_{\bfsigma}=0.0$.]{\includegraphics[width=.45\textwidth]{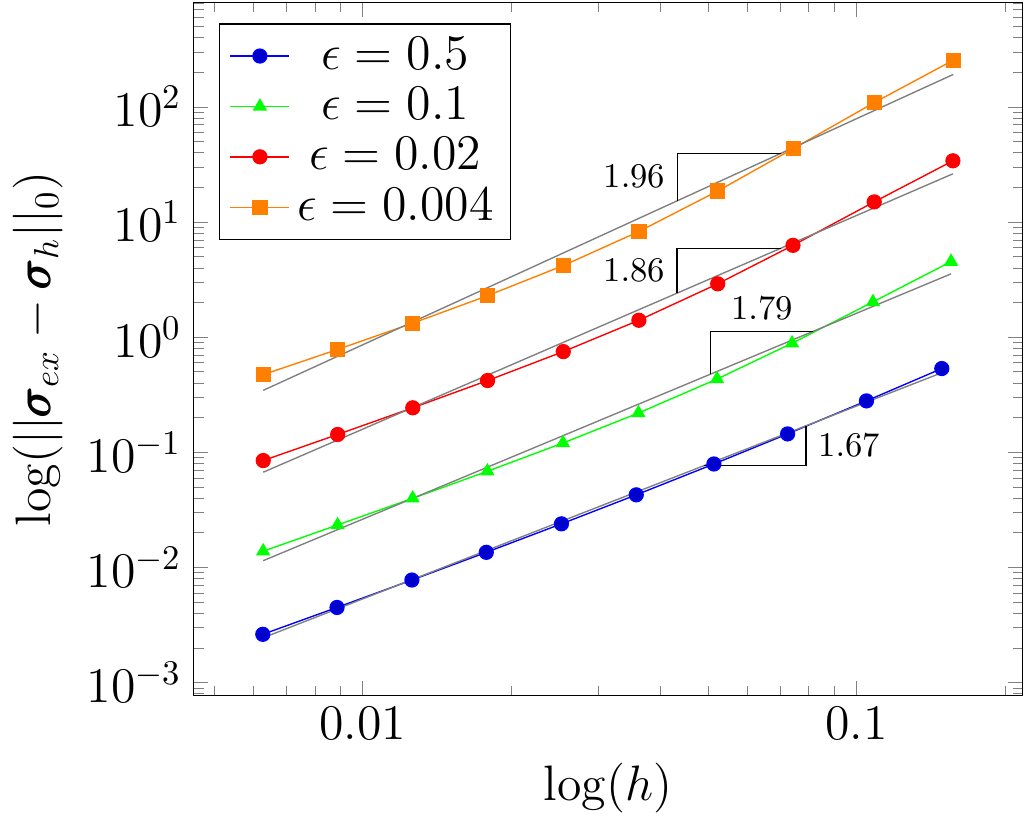}}
\caption{Convergence rates \protect\subref{subfig: conv sliver}, \protect\subref{subfig: stress w gammas} with extra-stress ghost penalty stabilization and  \protect\subref{subfig: conv sliver2},  \protect\subref{subfig: stress wo gammas} without extra-stress ghost penalty stabilization for $\gamma_b=15.0$, $\gamma_{\bfu}=0.1$ and $\gamma_p=0.1$.}
\label{fig: gammas ghost stab}
\end{figure}
\begin{figure}[htb]
\subfloat[ \label{subfig: sxx wo stab1}$\gamma_{\bfsigma}=0.0$, $\epsilon=0.02$, $\Delta x=0.077$.]{\includegraphics[width=.5\textwidth]{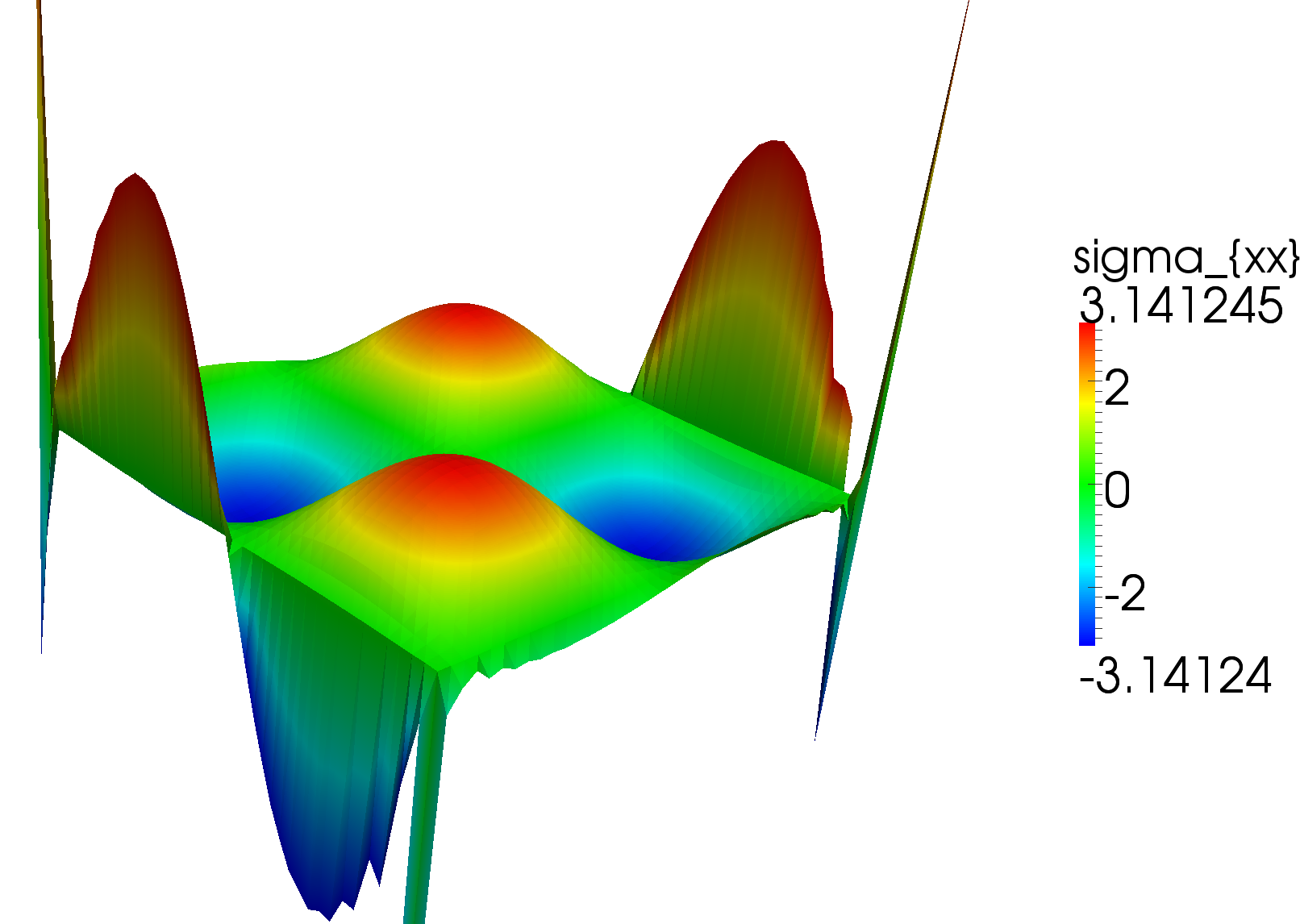}} 
\subfloat[ \label{subfig: sxx wo stab2} $\gamma_{\bfsigma}=0.0$, $\epsilon=0.02$, $\Delta x=0.018$. ]{\includegraphics[width=.5\textwidth]{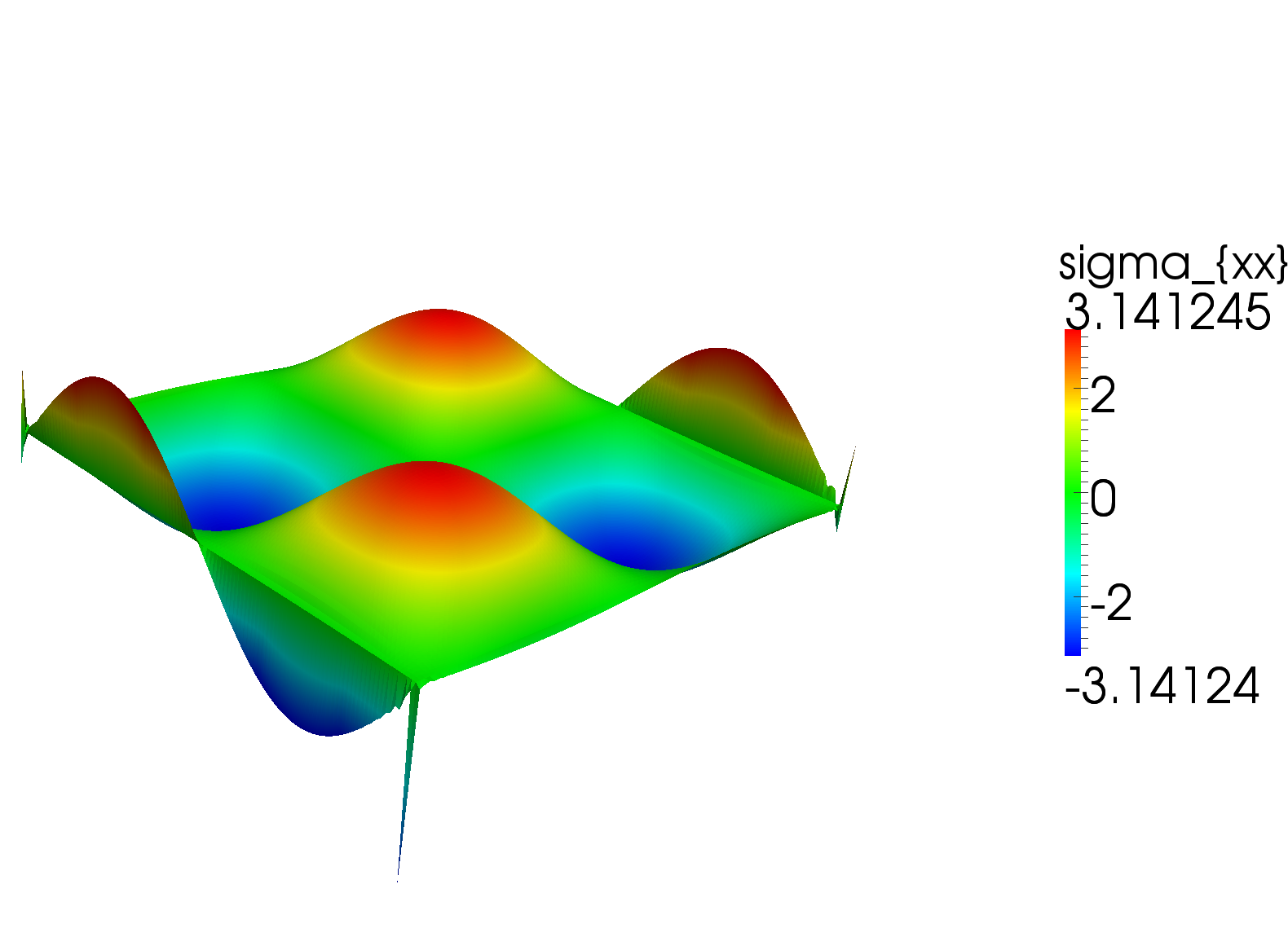}} \\
\subfloat[ \label{subfig: sxx w stab1}$\gamma_{\bfsigma}=0.1$, $\epsilon=0.02$, , $\Delta x=0.077$.]{\includegraphics[width=.5\textwidth]{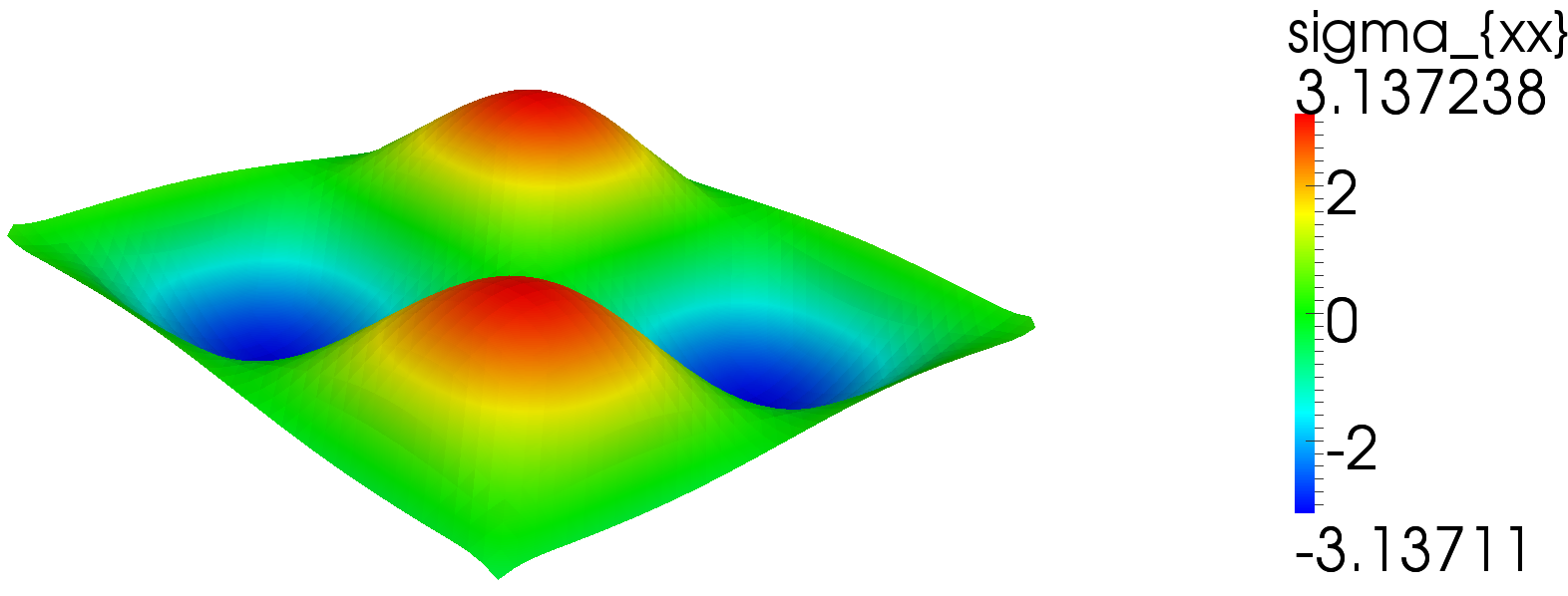}} 
\subfloat[ \label{subfig: sxx w stab2}$\gamma_{\bfsigma}=0.1$, $\epsilon=0.02$, $\Delta x=0.018$.]{\includegraphics[width=.5\textwidth]{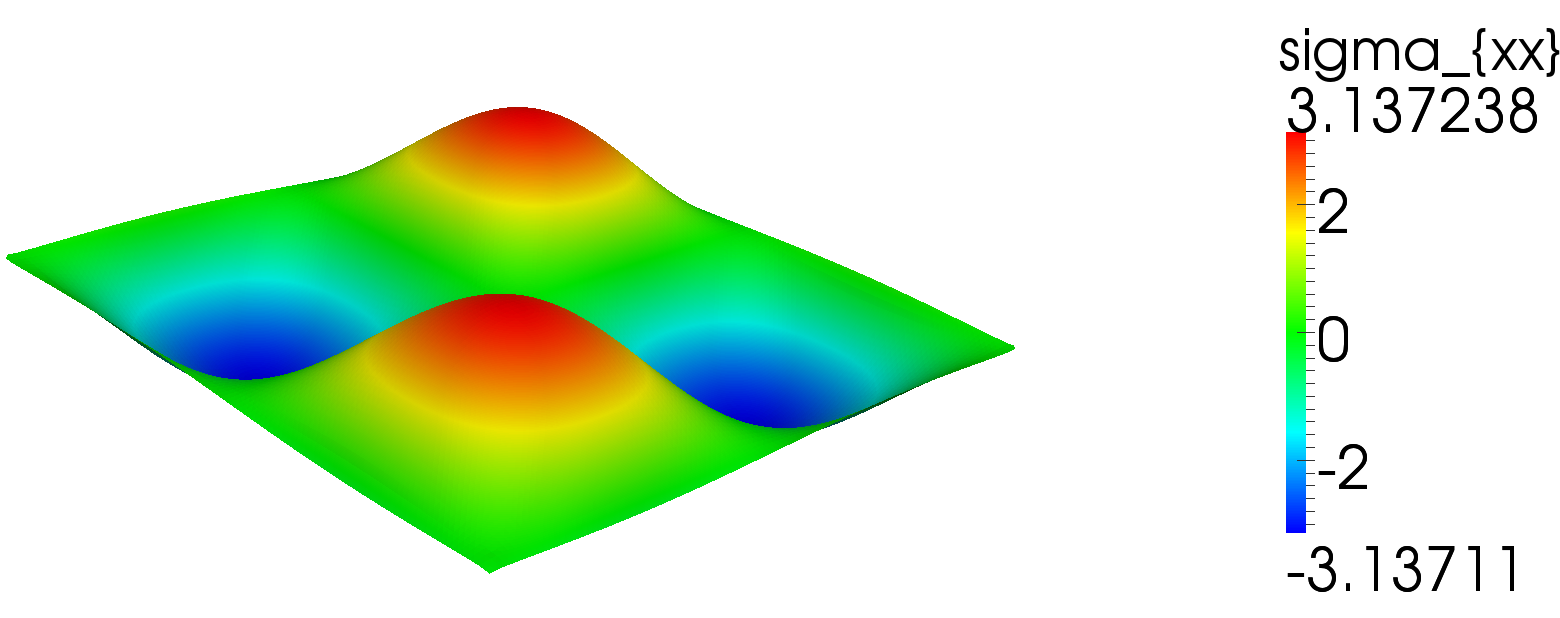}} 
\caption{Contour plot of $\sigma_{xx}$ component for \protect\subref{subfig: sxx wo stab1}, \protect\subref{subfig: sxx wo stab2} $\gamma_{\bfsigma}=0.0$ and \protect\subref{subfig: sxx w stab1}, \protect\subref{subfig: sxx w stab2}  $\gamma_{\bfsigma}=0.1$.}
\label{fig:contour-plots}
\end{figure}
\subsubsection{Condition number}
In this section, we investigate the condition number of the system
matrix $A$ ~\eqref{eq:A_h-definition} in dependence to the boundary
location for a ghost stabilization parameter of
$\gamma_{\bfsigma}=\left\{ 0.0, 0.001,0.1,1.0\right\}$. Here, we
consider a fixed fictitious domain $\Oast = [-1,1]^2$ with a fixed
mesh size $h$ and a shrinking physical domain
$\Omega=[-1+(1-\epsilon)\Delta x, 1-(1-\epsilon)\Delta x]^2$, where
$\Delta x$ is the edge length of the elements in $x$-direction and
$y$-direction. We choose  $\gamma_b=15.0, \gamma_u = 0.1, \gamma_p =
0.1$. Figure~\ref{fig: condition number} shows the condition number
with respect to the sliver size parameter $\epsilon$. We observe that
for $\gamma_{\bfsigma}=0.0$, the condition number is unbounded while
for $\gamma_{\bfsigma}=\left\{ 0.001,0.1,1.0\right\}$ the condition
number is bounded. Hence, even for very small ghost penalty
stabilization parameters the ill-conditioning dependence on the
boundary location is alleviated.
\begin{figure}[htb]
  \subfloat{ \includegraphics[width=0.25\textwidth]{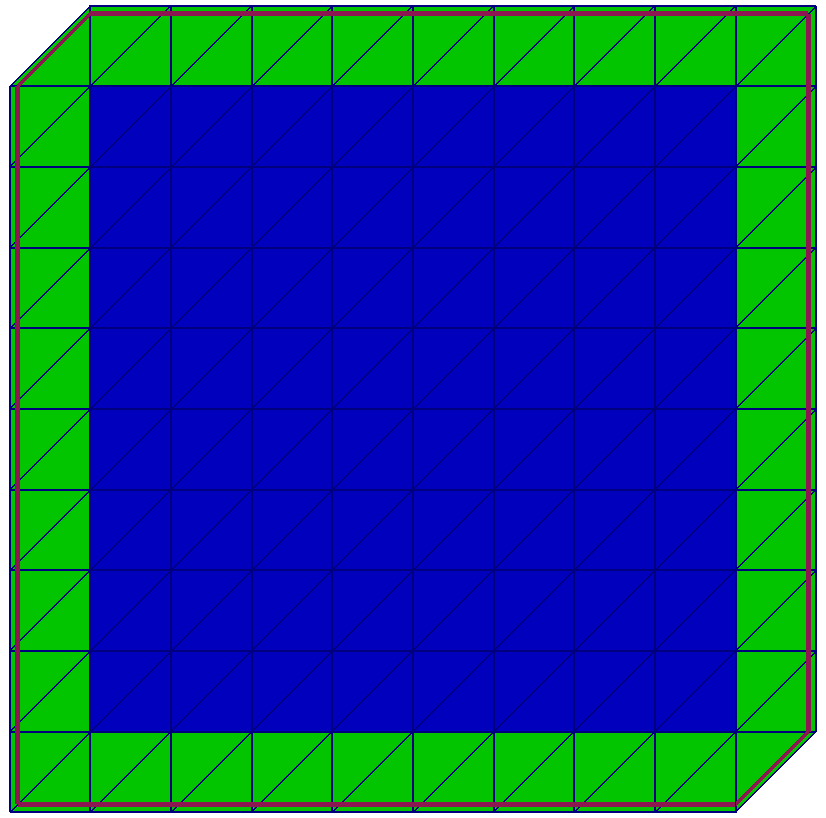}}
    \subfloat{    \includegraphics[width=0.25\textwidth]{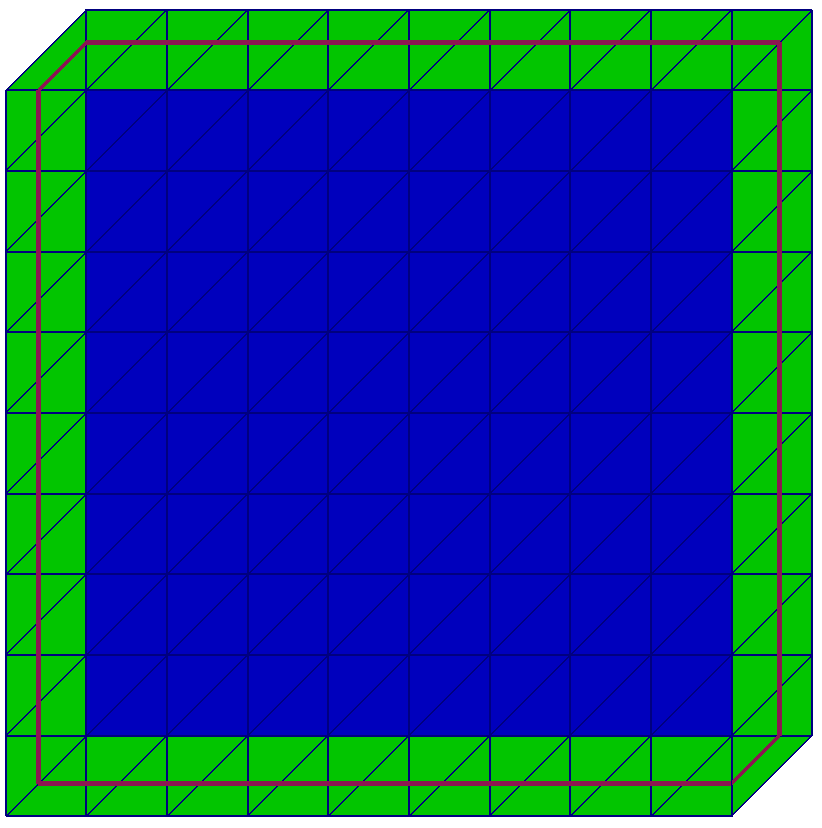}}
     \subfloat{   \includegraphics[width=0.25\textwidth]{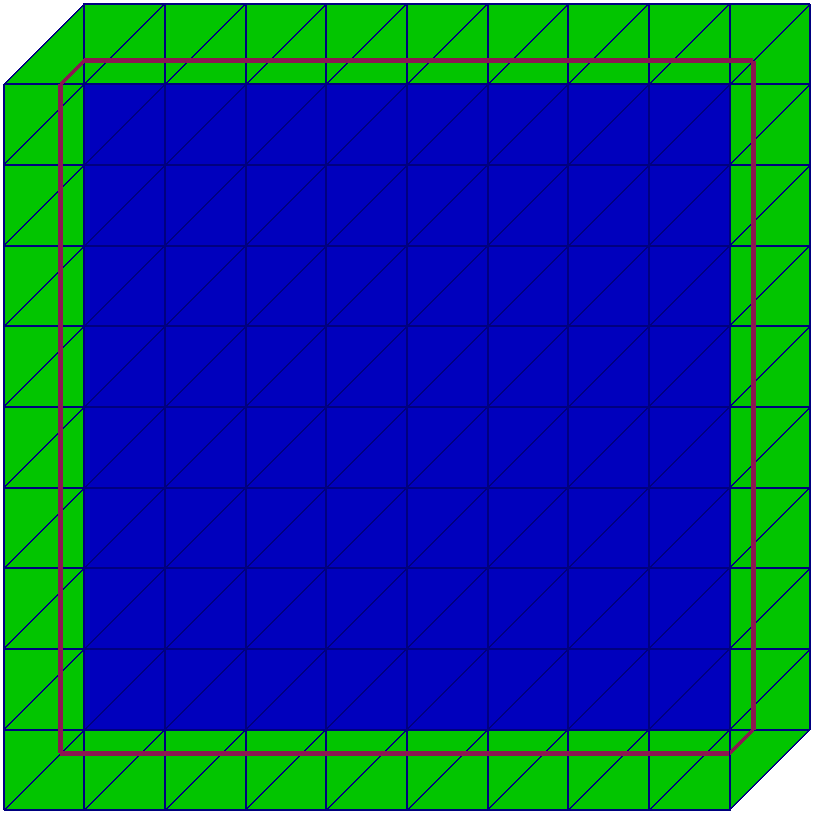}}
     \subfloat{   \includegraphics[width=0.25\textwidth]{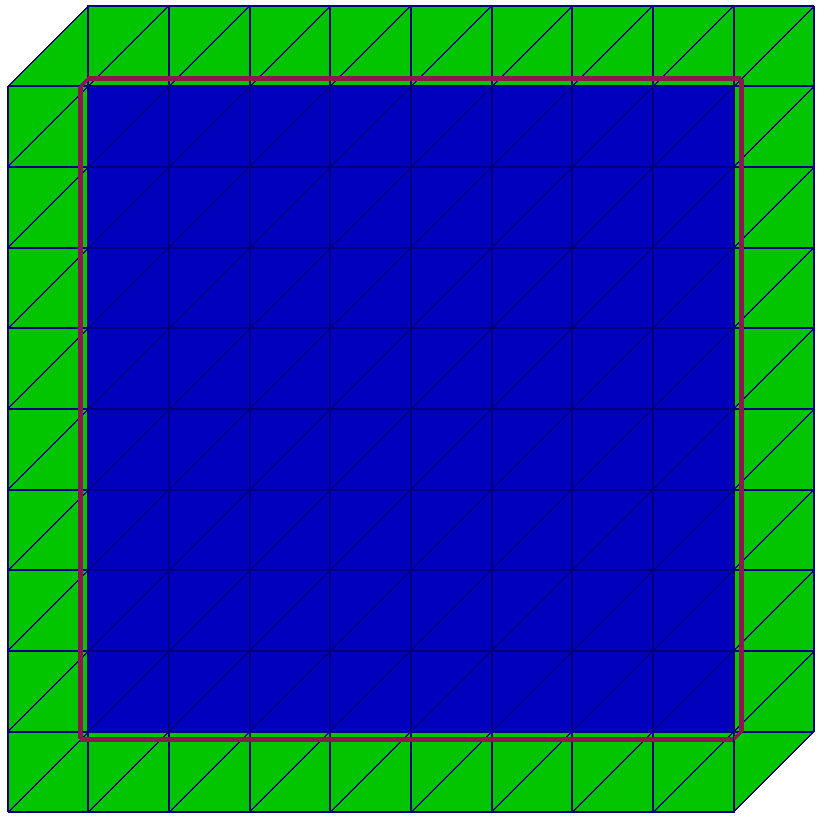}}\\
      \begin{center}
      \subfloat{   \includegraphics[width=.45\textwidth]{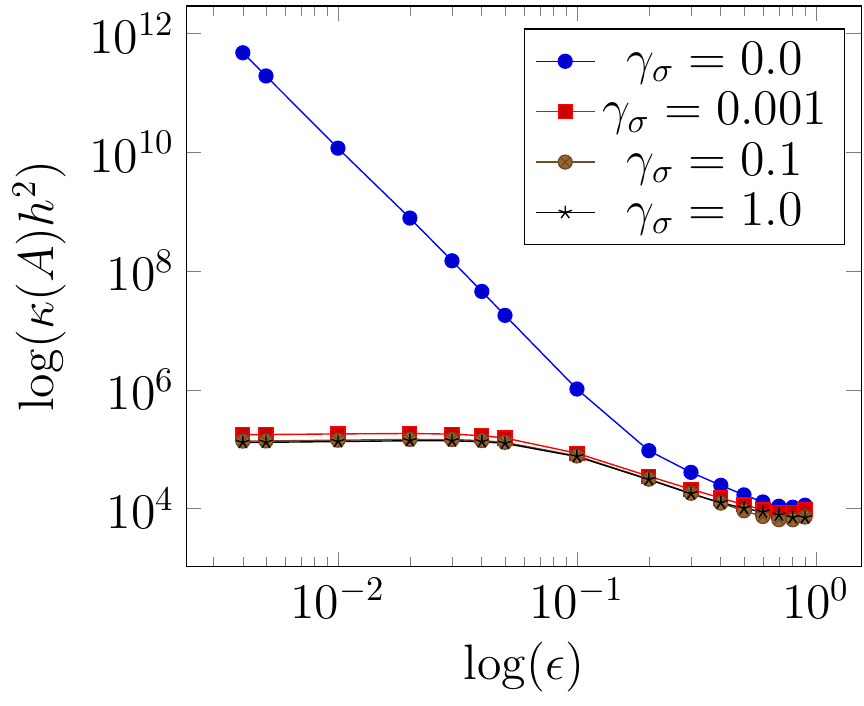}}
      \end{center}
\caption{Condition number $\kappa(A)$ for $\gamma_b=10.0, \gamma_u = 0.01, \gamma_p = 0.01$ and varying $\gamma_{\sigma}$ for a fixed fictitious domain $\Oast = [-1,1]^2$ with mesh size $h=0.2828$ ($\Delta x =0.2$) in terms of the sliver parameter $\epsilon$ for a shrinking physical domain $\Omega=[-1+(1-\epsilon)\Delta x, 1-(1-\epsilon)\Delta x]^2$.}  
\label{fig: condition number}
\end{figure}
\subsection{Three field Stokes in an aneurysm}
\label{p2:ssec:flow-in-complex-geometry}
As a final numerical example, we present
the computation of a fluid flow governed by the three field Stokes
problem in a three-dimensional domain with a complex
boundary geometry.
The boundary geometry is taken from a part of an arterial
network known as the Circle of Willis which is located close to the
human brain.  It is known that the network is prone to develop
aneurysms and therefore the computer-assisted study of the blood flow
in the Circle of Willis has been a recent subject of interest, see for
instance \citet{Steinman2003,IsaksenBazilevsKvamsdalEtAl2008,Valen-Sendstad2011}.\\
However, the purpose of this example is not to perform a realistic
study of the blood flow dynamics. Rather, we would like to demonstrate
the principal applicability of the developed method to simulation
scenarios where complex three-dimensional geometries are involved.
The blood vessel geometry is embedded in a structured background mesh
as illustrated in Figure~\ref{fig:aneurysm-box-mesh}.\\
The velocity is prescribed on the entire boundary $\Gamma$, 
where we set $\bfu = 0$ on the arterial walls and $\bfu = 1200\,
\mathrm{mm/s}$ on the inlet boundary. The two outflow velocities are
set such that the total flux is balanced. We choose $\eta= 1.0$,
$\gamma_{\bfu}= 0.1$, $\gamma_p= 0.1$, $\gamma_{\bfsigma}=0.1$ and $\gamma_b=10.0$.   \\
Figure~\ref{fig: aneurysm} displays the pressure, velocity and extra-stress profiles in the aneurysm geometry. The extra-stress tensor is displayed in terms of the van Mises stress measure \cite{Mises1913} given by 
\begin{equation}
\sigma_{v}^2 = \frac{1}{2} \left[ \left( \sigma_{xx} - \sigma_{yy}\right)^2 + \left( \sigma_{yy} - \sigma_{zz}\right)^2 + \left( \sigma_{zz} - \sigma_{xx}\right)^2  + 6 \left( \sigma_{xy}^2 + \sigma_{yz}^2 + \sigma_{xz}^2\right)\right].
\end{equation}
This stress measure provides an indication of the strength of normal stress differences and shear stresses in the fluid.
Although the fictitious domain mesh $\mesh$ provides only a coarse resolution of
the aneurysm geometry, the values of the velocity approximation
clearly conforms to the required boundary values on the actual surface
geometry.
\begin{figure}[htb]
\centering
  \subfloat[Aneurysm surface embedded in the structured background
  mesh. \label{fig:aneurysm-box-mesh}]{\includegraphics[width=.33\textwidth]{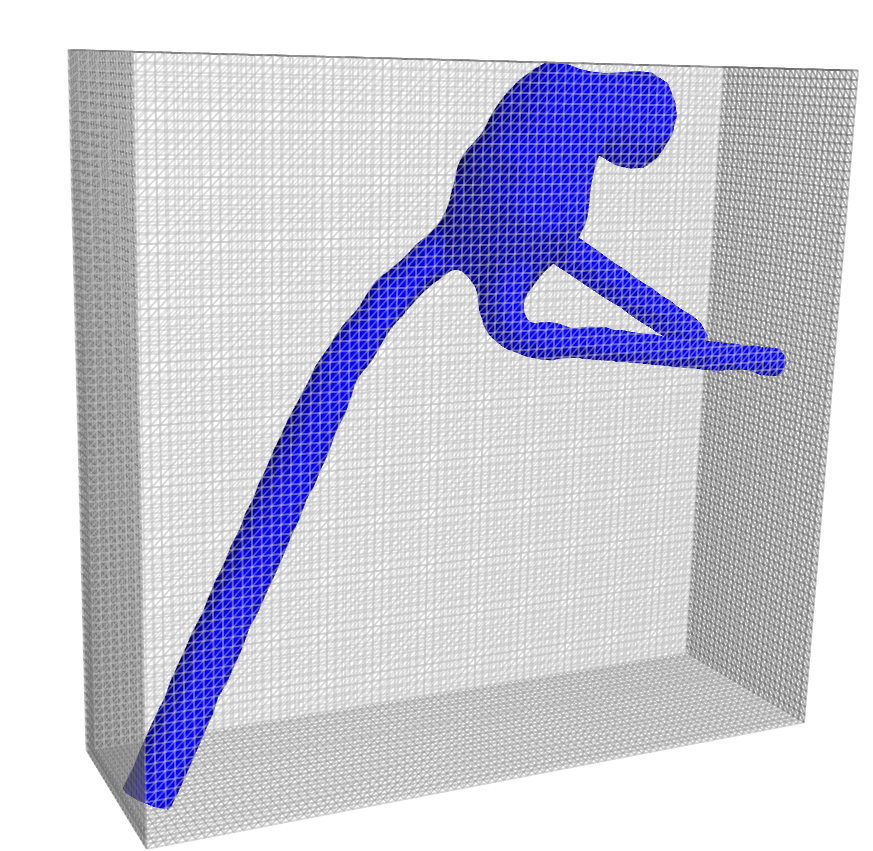}}   
  \subfloat[Pressure.]{\includegraphics[width=.33\textwidth]{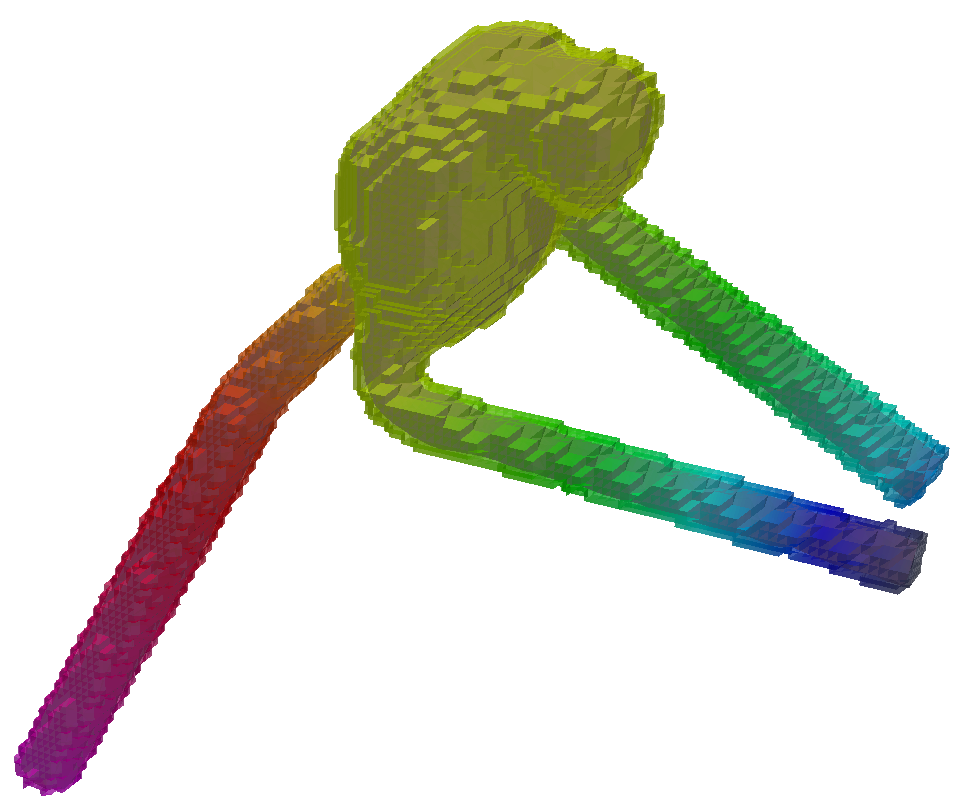}}   
  \subfloat{\includegraphics[width=.15\textwidth]{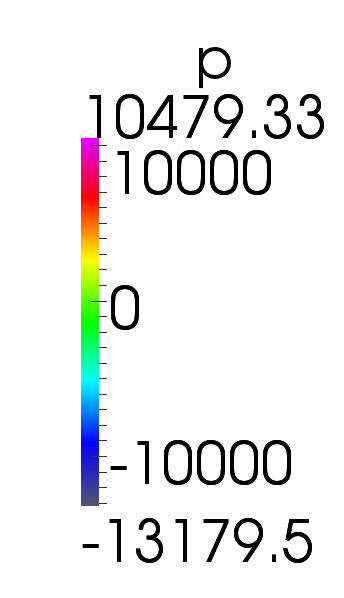}}  
  \\
  \subfloat[van Mises stress measure.]{\includegraphics[width=.33\textwidth]{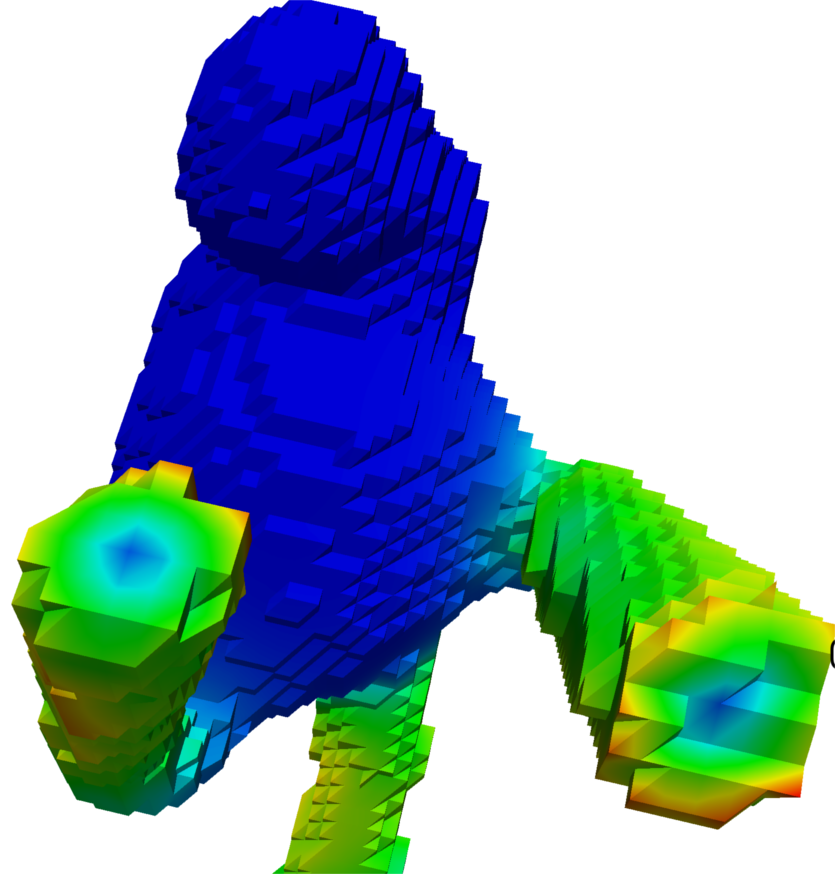}}     
  \subfloat{\includegraphics[width=.15\textwidth]{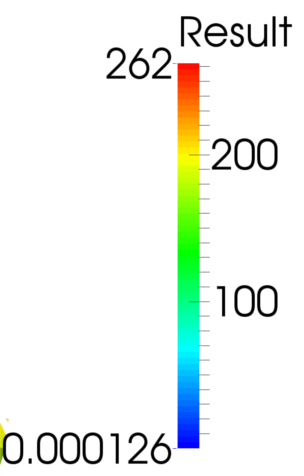}} 
  \subfloat[Velocity streamlines.]{\includegraphics[width=.4\textwidth]{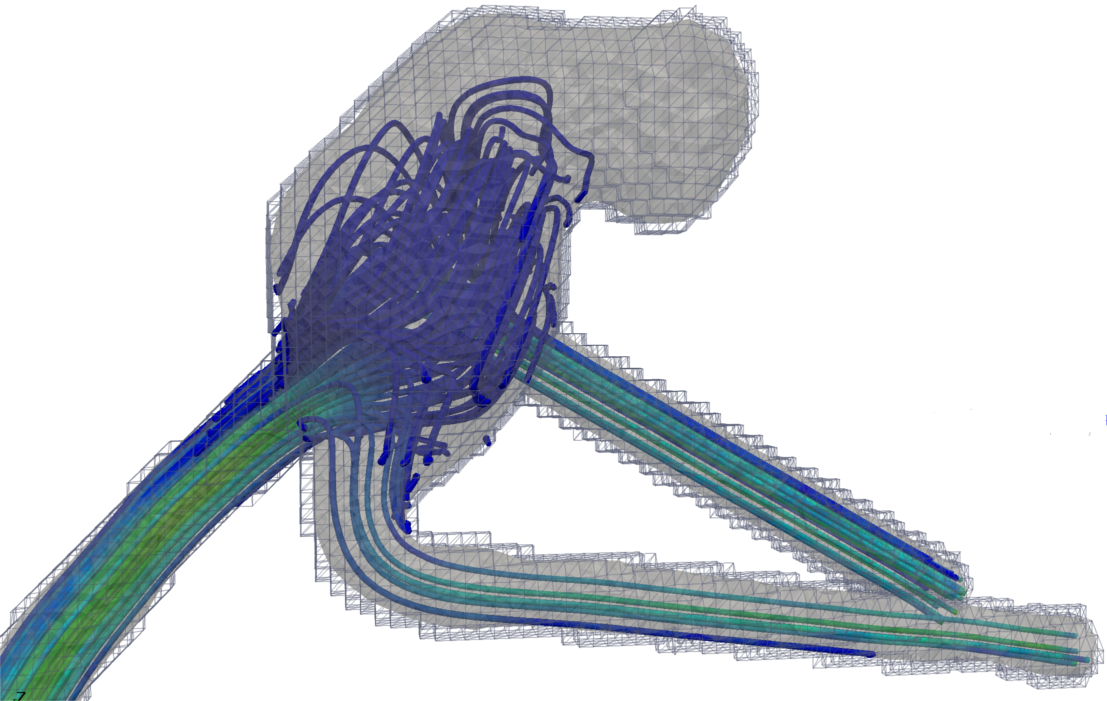}}   
  \subfloat{\includegraphics[width=.15\textwidth]{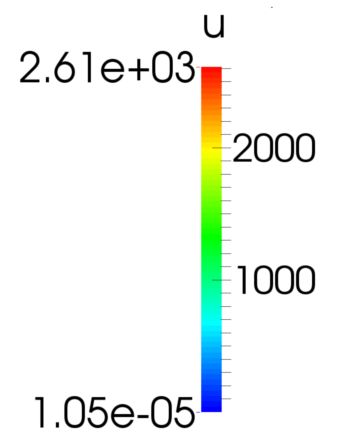}}  
\caption{Three field Stokes flow through an aneurysm.}
\label{fig: aneurysm}
\end{figure}

\section{Conclusions}
In this article, we have developed a novel fictitious domain method
for the three field Stokes equation. We have demonstrated
theoretically and numerically that our scheme is inf-sup stable and
possesses optimal convergence order properties independent of the
boundary location. We have approximated the velocity, pressure and
extra stress tensor with linear finite elements and we have stabilized
our scheme using a continuous interior penalty approach combined with
ghost penalty terms in the boundary region. We have demonstrated that
the ghost penalties in the boundary region guarantee a stable and
accurate solution independent of how the boundary intersects the mesh.
Additionally, we have demonstrated numerically that the ghost penalty
stabilization yields a bounded condition number independent of the
boundary location.  \\
In a future contribution, we will extend the scheme developed in this paper to multi-phase three field Stokes problems.
\clearpage

\section*{Acknowledgements}
The work for this article was supported by the EPSRC grant EP/J002313/2 on
"Computational methods for multiphysics interface problems" and a
Center of Excellence grant from the Research Council of Norway to the
Center for Biomedical Computing at Simula Research Laboratory.
The authors wish to thank Sebastian Warmbrunn for providing the
surface geometry used in
Section~\ref{p2:ssec:flow-in-complex-geometry}
and the anonymous referees for helpful comments.

\bibliographystyle{plainnat}

\begin{thebibliography}{31}
\providecommand{\natexlab}[1]{#1}
\providecommand{\url}[1]{\texttt{#1}}
\expandafter\ifx\csname urlstyle\endcsname\relax
  \providecommand{\doi}[1]{doi: #1}\else
  \providecommand{\doi}{doi: \begingroup \urlstyle{rm}\Url}\fi

\bibitem[Becker et~al.(2009)Becker, Burman, and Hansbo]{BeckerBurmanHansbo2009}
R.~Becker, E.~Burman, and P.~Hansbo.
\newblock {A Nitsche extended finite element method for incompressible
  elasticity with discontinuous modulus of elasticity}.
\newblock \emph{Comput. Methods Appl. Mech. Eng.}, 198\penalty0
  (41-44):\penalty0 3352--3360, 2009.

\bibitem[Bonito and Burman(2006)]{Bonito2006}
A.~Bonito and E.~Burman.
\newblock {A face penalty method for the three fields Stokes equation arising
  from Oldroyd-B viscoelastic flows}.
\newblock \emph{Numer. Math. Adv. Appl.}, 2:\penalty0 1--8, 2006.

\bibitem[Bonito and Burman(2008)]{Bonito2008}
A.~Bonito and E.~Burman.
\newblock {A Continuous Interior Penalty Method for Viscoelastic Flows}.
\newblock \emph{SIAM J. Sci. Comput.}, 30\penalty0 (3):\penalty0 1156--1177,
  2008.

\bibitem[Bonvin et~al.(2001)Bonvin, Picasso, and
  Stenberg]{BonvinPicassoStenberg2001}
J.~Bonvin, M.~Picasso, and R.~Stenberg.
\newblock {GLS and EVSS methods for a three-field Stokes problem arising from
  viscoelastic flows}.
\newblock \emph{Comput. Methods Appl. Mech. Eng.}, 190\penalty0 (29):\penalty0
  3893--3914, 2001.

\bibitem[Burman(2010)]{Burman2010}
E.~Burman.
\newblock {Ghost penalty}.
\newblock \emph{Comptes Rendus Mathematique}, 348\penalty0 (21-22):\penalty0
  1217--1220, 2010.

\bibitem[Burman and Hansbo(2006)]{BurmanHansbo2006a}
E.~Burman and P.~Hansbo.
\newblock {Edge stabilization for the generalized Stokes problem: A continuous
  interior penalty method}.
\newblock \emph{Comput. Methods Appl. Mech. Eng.}, 195\penalty0
  (19-22):\penalty0 2393--2410, 2006.

\bibitem[Burman and Hansbo(2010)]{BurmanHansbo2010}
E.~Burman and P.~Hansbo.
\newblock {Fictitious domain finite element methods using cut elements: I. A
  stabilized Lagrange multiplier method}.
\newblock \emph{Comput. Methods Appl. Mech. Eng.}, 199\penalty0 (41):\penalty0
  2680--2686, 2010.

\bibitem[Burman and Hansbo(2012)]{BurmanHansbo2012}
E.~Burman and P.~Hansbo.
\newblock {Fictitious domain finite element methods using cut elements: II. A
  stabilized Nitsche method}.
\newblock \emph{Appl. Numer. Math.}, 62\penalty0 (4):\penalty0 328--341, 2012.

\bibitem[Burman and Hansbo(2014)]{BurmanHansbo2013}
E.~Burman and P.~Hansbo.
\newblock Fictitious domain methods using cut elements: {III}. {A} stabilized
  {N}itsche method for {Stokes'} problem.
\newblock \emph{ESAIM, Math. Model. Num. Anal.}, 48\penalty0 (3):\penalty0
  859--874, 2014.

\bibitem[Burman et~al.(2006)Burman, Fern\'{a}ndez, and
  Hansbo]{BurmanFernandezHansbo2006}
E.~Burman, M.A. Fern\'{a}ndez, and P.~Hansbo.
\newblock {Continuous interior penalty finite element method for Oseen's
  equations}.
\newblock \emph{SIAM J. Numer. Anal.}, 44\penalty0 (3):\penalty0 1248--1274,
  2006.

\bibitem[Burman et~al.(2014{\natexlab{a}})Burman, Claus, Hansbo, Larson, and
  Massing]{BurmanClausHansboEtAl2014}
E.~Burman, S.~Claus, P.~Hansbo, M.G. Larson, and A.~Massing.
\newblock {CutFEM: discretizing geometry and partial differential equations}.
\newblock \emph{Int. J. Numer. Meth. Eng.}, 2014{\natexlab{a}}.

\bibitem[Burman et~al.(2014{\natexlab{b}})Burman, Hansbo, Larson, and
  Zahedi]{BurmanHansboLarsonEtAl2014}
E.~Burman, P.~Hansbo, M.G. Larson, and S.~Zahedi.
\newblock Cut finite element methods for coupled bulk-surface problems.
\newblock \emph{arXiv Prepr. arXiv:1403.6580}, 2014{\natexlab{b}}.

\bibitem[Di~Pietro and Ern(2011)]{DipietroErn2011}
D.~A. Di~Pietro and A.~Ern.
\newblock \emph{Mathematical aspects of discontinuous Galerkin methods}.
\newblock Springer, 2011.

\bibitem[Dolbow and Harari(2009)]{DoHa09}
J.~Dolbow and I.~Harari.
\newblock An efficient finite element method for embedded interface problems.
\newblock \emph{Int. J. Numer. Meth. Eng.}, 78\penalty0 (2):\penalty0 229--252,
  2009.

\bibitem[Donea et~al.(2004)Donea, Huerta, Ponthot, and
  Rodriguez-Ferran]{DoneaALE}
J.~Donea, A.~Huerta, J.-Ph. Ponthot, and A.~Rodriguez-Ferran.
\newblock \emph{{Arbitrary Lagrangian--Eulerian Methods}}, chapter~14.
\newblock John Wiley \& Sons Ltd., 2004.

\bibitem[Ern and Guermond(2004)]{ErnGuermond2004}
A.~Ern and J.-L. Guermond.
\newblock \emph{Theory and practice of finite elements}, volume 159 of
  \emph{Appl. Math. Sci.}
\newblock Springer, 2004.

\bibitem[Hansbo and Hansbo(2002)]{HansboHansbo2002}
A.~Hansbo and P.~Hansbo.
\newblock {An unfitted finite element method, based on Nitsche's method, for
  elliptic interface problems}.
\newblock \emph{Comput. Methods Appl. Mech. Eng.}, 191\penalty0
  (47-48):\penalty0 5537--5552, 2002.

\bibitem[Hansbo and Hansbo(2004)]{HaHa04}
A.~Hansbo and P.~Hansbo.
\newblock A finite element method for the simulation of strong and weak
  discontinuities in solid mechanics.
\newblock \emph{Comput. Methods Appl. Mech. Eng.}, 193\penalty0
  (33-35):\penalty0 3523--3540, 2004.

\bibitem[Hansbo et~al.(2003)Hansbo, Hansbo, and Larson]{HansboHansboLarson2003}
A.~Hansbo, P.~Hansbo, and M.~G. Larson.
\newblock {A Finite Element Method on Composite Grids based on Nitsche's
  Method}.
\newblock \emph{ESAIM, Math. Model. Num. Anal.}, 37\penalty0 (3):\penalty0
  495--514, 2003.

\bibitem[Harari and Dolbow(2010)]{HaDo10}
I.~Harari and J.~Dolbow.
\newblock Analysis of an efficient finite element method for embedded interface
  problems.
\newblock \emph{Comput. Mech.}, 46\penalty0 (1):\penalty0 205--211, 2010.

\bibitem[Isaksen et~al.(2008)Isaksen, Bazilevs, Kvamsdal, Zhang, Kaspersen,
  Waterloo, Romner, and Ingebrigtsen]{IsaksenBazilevsKvamsdalEtAl2008}
J.~G. Isaksen, Y.~Bazilevs, T.~Kvamsdal, Y.~Zhang, J.~H. Kaspersen,
  K.~Waterloo, B.~Romner, and T.~Ingebrigtsen.
\newblock {Determination of wall tension in cerebral artery aneurysms by
  numerical simulation}.
\newblock \emph{Stroke}, 39\penalty0 (12):\penalty0 3172, 2008.

\bibitem[Logg and Wells(2010)]{LoggWells2010a}
A.~Logg and G.~N. Wells.
\newblock {DOLFIN}: Automated finite element computing.
\newblock \emph{{ACM} Trans. Math. Softw.}, 37\penalty0 (2), 2010.

\bibitem[Logg et~al.(2012)Logg, Mardal, and et~al.]{LoggMardalEtAl2011}
A.~Logg, K.-A. Mardal, and Wells. G.~N. et~al.
\newblock \emph{Automated Solution of Differential Equations by the Finite
  Element Method}.
\newblock Springer, 2012.

\bibitem[Massing et~al.(2013)Massing, Larson, Logg, and
  Rognes]{MassingLarsonLoggEtAl2013a}
A.~Massing, M.~G. Larson, A.~Logg, and M.~E. Rognes.
\newblock {A stabilized Nitsche fictitious domain method for the Stokes
  problem}.
\newblock \emph{J. Sci. Comput.}, pages 1--25, 2013.

\bibitem[Massing et~al.(2014)Massing, Larson, Logg, and
  Rognes]{MassingLarsonLoggEtAl2013}
A.~Massing, M.G. Larson, A.~Logg, and M.E. Rognes.
\newblock {A stabilized Nitsche overlapping mesh method for the Stokes
  problem}.
\newblock \emph{Num. Math.}, pages 1--29, 2014.
\newblock \doi{10.1007/s00211-013-0603-z}.

\bibitem[Mises(1913)]{Mises1913}
R.~V. Mises.
\newblock {Mechanik der festen K\"{o}rper im plastisch--deformablen Zustand}.
\newblock \emph{{Nachrichten von der Gesellschaft der Wissenschaften zu
  G\"{o}ttingen, Mathematisch--Physikalische Klasse}}, pages 582--592, 1913.

\bibitem[Nitsche(1971)]{Nitsche1971}
J.~Nitsche.
\newblock {\"{U}ber ein Variationsprinzip zur L\"{o}sung von
  Dirichlet-Problemen bei Verwendung von Teilr\"{a}umen, die keinen
  Randbedingungen unterworfen sind}.
\newblock \emph{Abhandlungen aus dem Mathematischen Seminar der Universit\"{a}t
  Hamburg}, 36\penalty0 (1):\penalty0 9--15, July 1971.

\bibitem[Owens and Phillips(2002)]{Owens2002}
R.~G. Owens and T.~N. Phillips.
\newblock \emph{{Computational Rheology}}.
\newblock World Scientific, 2002.

\bibitem[Stein(1970)]{Stein1970}
E.~Stein.
\newblock \emph{{Singular Integrals and Differentiability Properties of
  Functions}}.
\newblock Princeton University Press, 1970.

\bibitem[Steinman et~al.(2003)Steinman, Milner, Norley, Lownie, and
  Holdsworth]{Steinman2003}
D.~A. Steinman, J.~S. Milner, C.~J. Norley, S.~P. Lownie, and D.~W. Holdsworth.
\newblock {Image-based computational simulation of flow dynamics in a giant
  intracranial aneurysm.}
\newblock \emph{AJNR. American journal of neuroradiology}, 24\penalty0
  (4):\penalty0 559--66, April 2003.

\bibitem[Valen-Sendstad et~al.(2011)Valen-Sendstad, Mardal, Mortensen, Reif,
  and Langtangen]{Valen-Sendstad2011}
K.~Valen-Sendstad, K.~Mardal, M.~Mortensen, B.~A.~P. Reif, and H.~P.
  Langtangen.
\newblock {Direct numerical simulation of transitional flow in a
  patient-specific intracranial aneurysm.}
\newblock \emph{J Biomech}, 44\penalty0 (16):\penalty0 2826--32, 2011.

\end{thebibliography}

\end{document}